\documentclass[onefignum,onetabnum]{siamart190516}

\usepackage{amssymb,latexsym,enumerate,verbatim,amsfonts}
\usepackage{color} 
\usepackage{graphicx,subcaption,microtype}
\usepackage{algorithmic}
\usepackage[active]{srcltx}
\usepackage{multirow}
\usepackage{rotating}
\usepackage{booktabs}
\usepackage[normalem]{ulem}
\usepackage{xcolor}
\allowdisplaybreaks
\usepackage{cite}

\newif\ifistoreview
\istoreviewtrue

\newcommand{\addColor}{\textcolor{black}}
\newcommand{\add}[1]{\ifistoreview\addColor{#1}\else #1\fi}

\def\R{{\rm I\!R}}
\def\argmin{\mathop{\rm arg\,min}}

\graphicspath{{Figures/}} 

\newsiamthm{assumption}{Assumption}
\newsiamthm{remark}{Remark}

\headers{Analysis and Algorithms for L1/L2 minimization}{Liaoyuan Zeng, Peiran Yu and Ting Kei Pong}

\title{Analysis and algorithms for some compressed sensing models based on L1/L2 minimization
\thanks{ \funding{The research of the first author is supported in part by the AMSS-PolyU Joint Laboratory Postdoctoral Scheme. The research of the third author is supported in part by Hong Kong Research Grants Council PolyU153003/19p.} } }

\author{Liaoyuan Zeng\thanks{Department of Applied Mathematics, The Hong Kong Polytechnic University, Hong Kong, PRC
  (\email{lyzeng@polyu.edu.hk}).}
\and
Peiran Yu\thanks{Department of Applied Mathematics, The Hong Kong Polytechnic University, Hong Kong, PRC
  (\email{peiran.yu@connect.polyu.hk}).}
\and Ting Kei Pong\thanks{Department of Applied Mathematics, The Hong Kong Polytechnic University, Hong Kong, PRC
  (\email{tk.pong@polyu.edu.hk}).}
}

\ifpdf
\hypersetup{
  pdftitle={Analysis and Algorithms for L1/L2 minimization},
  pdfauthor={Liaoyuan Zeng and Peiran Yu and Ting Kei Pong}
}
\fi

\begin{document}

\maketitle

\begin{abstract}
  Recently, in a series of papers \cite{YinEssXin14,RaWaDoLo18,WangYanLou19,WangTaoLou20}, the ratio of $\ell_1$ and $\ell_2$ norms was proposed as a sparsity inducing function for noiseless compressed sensing. In this paper, we further study properties of such model in the noiseless setting, and propose an algorithm for minimizing $\ell_1$/$\ell_2$ subject to noise in the measurements. Specifically, we show that the extended objective function (the sum of the objective and the indicator function of the constraint set) of the model in \cite{RaWaDoLo18} satisfies the Kurdyka-{\L}ojasiewicz (KL) property with exponent 1/2; this allows us to establish linear convergence of the algorithm proposed in \cite[Eq.~11]{WangYanLou19} under mild assumptions. We next extend the $\ell_1$/$\ell_2$ model to handle compressed sensing problems with noise. We establish the solution existence for some of these models under the spherical section property \cite{Vavasis09,Zhang13}, and extend the algorithm in \cite[Eq.~11]{WangYanLou19} by incorporating moving-balls-approximation techniques \cite{AusSheTeb10} for solving these problems. We prove the subsequential convergence of our algorithm under mild conditions, and establish global convergence of the whole sequence generated by our algorithm by imposing additional KL and differentiability assumptions on a specially constructed potential function. Finally, we perform numerical experiments on robust compressed sensing and basis pursuit denoising with residual error measured by $ \ell_2 $ norm or Lorentzian norm via solving the corresponding $\ell_1$/$\ell_2$ models by our algorithm. Our numerical simulations show that our algorithm is able to recover the original sparse vectors with reasonable accuracy.
\end{abstract}



\section{Introduction}
In compressed sensing (CS), a high-dimensional sparse or approximately sparse signal $ x_0\in \R^n$ is compressed (linearly) as $Ax_0$ for transmission, where $A\in \R^{m\times n}$ is the sensing matrix. The CS problem seeks to recover the original signal $ x_0 $ from the possibly noisy low-dimensional measurement $ b \in \R^m$. This problem is NP-hard in general; see \cite{Natarajan95}.

When there is no noise in the transmission, i.e., $Ax_0 = b$, one can recover $x_0$ exactly by minimizing the $\ell_1$ norm over $A^{-1}\{b\}$ if $x_0$ is sufficiently sparse and the matrix $A$ satisfies certain assumptions \cite{CaTao05, ChenDonoho01}. To empirically enhance the recovery ability, various nonconvex models like $ \ell_p $ $(0<p<1)$ minimization model \cite{Chartrand07} and $ \ell_{1-2} $ minimization model \cite{LouYin15} have been proposed, in which the $\ell_p$ quasi-norm and the difference of $\ell_1$ and $\ell_2$ norms are minimized over $A^{-1}\{b\}$, respectively. Recently, a new nonconvex model based on minimizing the quotient of the $\ell_1$ and $\ell_2$ norms was introduced in \cite{RaWaDoLo18,YinEssXin14} and further studied in \cite{WangYanLou19,WangTaoLou20}:
\begin{equation}\label{Pold}
\nu_{cs}^*:=\displaystyle\min_{x\in \R^n}  \displaystyle\frac{\|x\|_1}{\|x\|} \quad\quad {\rm s.t.}  \quad Ax = b,
\end{equation}
where $A\in \R^{m\times n}$ has full row rank and $b\in \R^m\backslash\{0\}$. As discussed in \cite{RaWaDoLo18}, the above $ \ell_1/\ell_2 $ model has the advantage of being \emph{scale-invariant} when reconstructing signals and images with high dynamic range.  An efficient algorithm was proposed for solving \eqref{Pold} in \cite[Eq.~11]{WangYanLou19} and subsequential convergence was established under mild assumptions. 


In practice, however, there is noise in the measurement, i.e., $b = Ax_0 + \epsilon$ for some noise vector $\epsilon$, and \eqref{Pold} is not applicable for (approximately) recovering $x_0$. To deal with noisy situations, it is customary to relax the equality constraint in \eqref{Pold} to an inequality constraint \cite{CaTao06}. In this paper, we consider the following model that minimizes the $ \ell_1/\ell_2 $ objective over an inequality constraint:
\begin{equation}\label{P0}
\displaystyle  \nu_{ncs}^* =\min_{x\in \R^n}  \displaystyle\frac{\|x\|_1}{\|x\|} \quad \quad {\rm s.t.}\   q(x)\le 0,
\end{equation}
where $q(x) = P_1(x) - P_2(x)$ with $ P_1:\R^n\to \R $ being \add{continuously differentiable with Lipschitz continuous gradient} and $ P_2:\R^n\to \R $ being convex continuous, and we assume that $\{x:\; q(x)\le 0\}\neq\emptyset$ and $q(0) > 0$. Our assumptions on $q$ are general enough to cover commonly used loss functions for modeling noise in various scenarios:
\begin{enumerate}
	\item {\bf Gaussian noise}: When the noise in the measurement follows the Gaussian distribution, the least squares loss function $ y\mapsto \|y-b\|^2 $ is typically employed \cite{ChenDonoho01,CaTao06}. One may consider the following $\ell_1/\ell_2$ minimization problem:
	\begin{equation}\label{prob:LS}
	\displaystyle \min_{x\in \R^n}  \displaystyle\frac{\|x\|_1}{\|x\|} \quad\quad
	{\rm s.t.} \ \|Ax - b\|^2 - \sigma^2 \le 0,
	\end{equation}
	where $\sigma > 0$, $A\in \R^{m\times n}$ has full row rank and $b\in \R^m$ satisfies $ \|b\|>\sigma $. Problem \eqref{prob:LS} corresponds to \eqref{P0} with $q(x) = P_1(x) = \|Ax - b\|^2 - \sigma^2$ and $P_2 = 0$.
	\item {\bf Cauchy noise}: When the noise in the measurement follows the Cauchy distribution (a heavy-tailed distribution), the Lorentzian norm\footnote{\add{We refer the readers to \cite[Equation~(12)]{CaRaArBaSa16} for the definition and notation of Lorentzian norm.}} $ \|y\|_{LL_2,\gamma} := \sum_{i=1}^m\log\left(1 + \gamma^{-2}y_i^2\right) $ is used as the loss function \cite{CaBa10,CaRaArBaSa16}, where $\gamma > 0$. Note that the Lorentzian norm is \add{continuously differentiable with Lipschitz continuous gradient}. One may then consider the following $\ell_1/\ell_2$ minimization problem:
	\begin{equation}\label{prob:Lorentzian}
	\displaystyle\min_{x\in \R^n}  \displaystyle\frac{\|x\|_1}{\|x\|} \quad\quad
	{\rm s.t.} \ \|Ax - b\|_{LL_2,\gamma} - \sigma\le 0,
	\end{equation}
	where $\sigma > 0$, $A\in \R^{m\times n}$ has full row rank, and $ b\in\R^m $ with $ \|b\|_{LL_2,\gamma}>\sigma $. Problem \eqref{prob:Lorentzian} corresponds to \eqref{P0} with $q(x) = P_1(x)= \|Ax - b\|_{LL_2,\gamma} - \sigma$ and $P_2 = 0$.
	\item {\bf  Robust compressed sensing}: In this scenario, the measurement is corrupted by both Gaussian noise and electromyographic noise \cite{PoCa12,CaRaArBaSa16}: the latter is sparse and may have large magnitude (outliers). Following \cite[Section~5.1.1]{LiuPong19}, one may make use of the loss function $y\mapsto {\rm dist}^2(y,S)$, where $ S:=\{z\in\R^m:\; \|z\|_0\leq r\} $, $\|z\|_0$ is the number of nonzero entries in $z$ and $r$ is an estimate of the number of outliers. One may then consider the following $\ell_1/\ell_2$ minimization problem:
	\begin{equation}\label{prob:rcs}
	\displaystyle\min_{x\in \R^n}  \displaystyle\frac{\|x\|_1}{\|x\|}\quad\quad
	{\rm s.t.} \ {\rm dist}^2(Ax-b, S)-\sigma^2  \le 0,
	\end{equation}
	where $ \sigma>0 $, $ S=\{z\in\R^m:\; \|z\|_0\leq r\} $ with $r \ge 0$,  $A\in\R^{m\times n} $ has full row rank and $ b\in\R^m $ satisfies $ {\rm dist}(b, S)>\sigma $. Notice that
	\begin{equation}\label{q_rcs}
	{\rm dist}^2(Ax-b, S)-\sigma^2 = \underbrace{\|Ax-b\|^2-\sigma^2}_{P_1(x)} - \underbrace{\max_{z\in S} \{\langle 2z ,Ax-b\rangle - \|z\|^2 \} }_{P_2(x)},
	\end{equation}
	with $ P_1 $ being \add{continuously differentiable with Lipschitz continuous gradient} and $ P_2 $ being convex continuous. So this problem corresponds to \eqref{P0} with $ P_1 $ and $ P_2 $ as in \eqref{q_rcs} and $ q = P_1-P_2 $.
\end{enumerate}
In the literature, algorithms for solving \eqref{prob:LS} with $ \ell_1 $ norm or $ \ell_p $ quasi-norm in place of the quotient of the $\ell_1$ and $\ell_2$ norms have been discussed in \cite{VaFr08,ChenPong16,SheTeb16}, and \cite{YuPongLv20} discussed an algorithm for solving \eqref{prob:Lorentzian} with $ \ell_1 $ norm in place of the quotient of the $\ell_1$ and $\ell_2$ norms. These existing algorithms, however, are not directly applicable for solving \eqref{P0} due to the fractional objective and the possibly nonsmooth continuous function $q$ in the constraint.

In this paper, we further study properties of the $\ell_1/\ell_2$ models \eqref{Pold} and \eqref{P0}, and propose an algorithm for solving \eqref{P0}. In particular, we first argue that an optimal solution of \eqref{Pold} exists by making connections with the $s$-spherical section property \cite{Vavasis09,Zhang13} of $\ker A$: a property which is known to hold with high probability when $n\gg m$ for Gaussian matrices. We then revisit the algorithm proposed in \cite[Eq.~11]{WangYanLou19} (see Algorithm~\ref{alg:lou} below) for solving \eqref{Pold}. Specifically, we consider the following function
\begin{equation}\label{Fnoiseless}
F(x) := \frac{\|x\|_1}{\|x\|} + \delta_{A^{-1}\{b\}}(x),
\end{equation}
where $A\in \R^{m\times n}$ has full row rank and $b\in \R^m\backslash\{0\}$. We show in section~\ref{sec42} that $F$ is a Kurdyka-{\L}ojasiewicz (KL) function with exponent $\frac12$. This together with standard convergence analysis based on KL property \cite{Attouch09,Attouch10,Attouch13} allows us to deduce local linear convergence of the sequence $\{x^t\}$ generated by Algorithm~\ref{alg:lou} when $\{x^t\}$ is bounded. The KL exponent of $F$ is obtained based on a new calculus rule that deduces the KL exponent of a fractional objective from the difference between the numerator and (a suitable scaling of) the denominator.

Next, for the model \eqref{P0}, we also relate existence of solutions to the $s$-spherical section property of $\ker A$ when $q$ takes the form in \eqref{prob:LS} and \eqref{prob:Lorentzian}. We then propose an algorithm, which we call MBA$_{\ell_1/\ell_2}$ (see Algorithm~\ref{alg1}), for solving \eqref{P0}, which can be seen as an extension of Algorithm~\ref{alg:lou} by incorporating {\em moving-balls-approximation} (MBA) techniques. The MBA algorithm was first proposed in \cite{AusSheTeb10} for minimizing a smooth objective function subject to multiple smooth constraints, and was further studied in \cite{BolChenPau18,BolPau16,YuPongLv20} for more general objective functions. However, the existing convergence results of these algorithms cannot be applied to MBA$_{\ell_1/\ell_2}$ because of the possibly nonsmooth continuous function $q$ and the fractional objective in \eqref{P0}. Our convergence analysis of MBA$_{\ell_1/\ell_2}$ relies on a specially constructed potential function, which involves the indicator function of the lower level set of a {\em proper closed} function related to $q$ (see \eqref{barF}). We prove that any accumulation point of the sequence generated by MBA$_{\ell_1/\ell_2}$ is a so-called Clarke critical point, under mild assumptions; Clarke criticality reduces to the usual notion of stationarity when $q$ is regular. Moreover, by imposing  additional KL assumptions on this potential function and assuming $P_1$ is twice continuously differentiable, we show that the sequence generated by MBA$_{\ell_1/\ell_2}$ is globally convergent, and the convergence rate is related to the KL exponent of the potential function.
Finally, we perform numerical experiments to illustrate the performance of our algorithm on solving \eqref{prob:LS}, \eqref{prob:Lorentzian} and \eqref{prob:rcs}.

The rest of this paper is organized as follows. We present notation and some preliminaries in section~\ref{sec2}. Existence of solutions for \eqref{Pold} is discussed in section~\ref{sec3}. In section~\ref{sec4}, we derive the KL exponent of $F$ in \eqref{Fnoiseless} and establish local linear convergence of the algorithm proposed in \cite[Eq.~11]{WangYanLou19}. Properties of \eqref{P0} such as solution existence and optimality conditions are discussed in section~\ref{sec:noisyCS}, and our algorithm, MBA$_{\ell_1/\ell_2}$, for solving \eqref{P0} is proposed and analyzed in section~\ref{sec6}. Finally, numerical results are presented in section~\ref{sec7}.

\section{Notation and preliminaries}\label{sec2}
In this paper, we use $ \R^n $ to denote the Euclidean space of dimension $ n $ and use $ \mathbb{N}_+ $ to denote the set of nonnegative integers. For two vectors $ x$ and $y\in\R^n $, we use $\langle x,y\rangle$ to denote their inner product, i.e., $\langle x,y\rangle = \sum_{i=1}^n x_i y_i$. The Euclidean norm, the $\ell_1$ norm and the $\ell_0$ norm (i.e., the number of nonzero entries) of $x$ are denoted respectively by $ \|x\| $, $ \|x\|_1 $ and $\|x\|_0$. We also use $ B(x, r) $ to denote a closed ball centered at $ x $ with radius $r\ge 0$, i.e., $ B(x,r) = \{y:\; \|y-x\|\leq r\} $.

An extended-real-valued function $ h:\R^n\rightarrow (-\infty, \infty] $ is said to be proper if its domain $ {\rm dom}\, h:=\{x:h(x)< \infty\} $ is nonempty. A proper function $ h $ is said to be closed if it is lower semi-continuous. For a proper closed function $ h $, the regular subdifferential $ \widehat\partial h(\bar x) $ at $\bar x\in {\rm dom}\, h $ and the limiting subdifferential $ \partial h(\bar x) $ are given respectively as
\begin{equation*}
	\begin{split}
	& \widehat{\partial}h(\bar x):= \left\{\upsilon:\; \liminf_{x\rightarrow \bar x, x\neq \bar x} \frac{h(x)-h(\bar x)-\langle\upsilon, x-\bar x\rangle}{\|x-\bar x\|}\geq 0\right\}, \\
	& \partial h(\bar x):= \left\{\upsilon:\; \exists x^t\overset{h}\rightarrow \bar x \text{ and } \upsilon^t \in \widehat{\partial} h(x^t) \text{ with } \upsilon^t \rightarrow \upsilon\right\},
	\end{split}
\end{equation*}
where $ x^t\overset{h}\rightarrow \bar x $ means $ x^t\rightarrow \bar x $ and $ h(x^t)\rightarrow h(\bar x) $. In addition, we set $\partial h(x) = \widehat\partial h(x) = \emptyset$ by convention when $x\notin {\rm dom}\,h$, and we define ${\rm dom}\,\partial h:= \{x:\; \partial h(x)\neq\emptyset\}$. It is known that $\partial h(x) = \{\nabla h(x)\}$ if $h$ is continuously differentiable at $x$ \cite[Exercise~8.8(b)]{RockWets98}. For a proper closed convex function, the limiting subdifferential reduces to the classical subdifferential for convex functions \cite[Proposition~8.12]{RockWets98}.
The convex conjugate of a proper closed convex function $h$ is defined as
\[
h^*(y)=\sup_{x\in\R^n} \{ \langle x,y\rangle- h(x) \}.
\]
We recall the following relationship concerning convex conjugate and subdifferential of a proper closed convex function $h$; see \cite[Proposition~11.3]{RockWets98}:
\begin{equation}\label{Young}
  y \in \partial h(x)\ \Leftrightarrow\ x\in \partial h^*(y) \ \Leftrightarrow\  h(x) + h^*(y) \le \langle x,y\rangle  \ \Leftrightarrow\  h(x) + h^*(y) = \langle x,y\rangle.
\end{equation}
For a locally Lipschitz function $h$, its Clarke subdifferential at $ \bar x\in\R^n $ is defined by
\[
\partial^\circ h(\bar x):= \left\{\upsilon :\; \limsup_{x\rightarrow \bar x, t\downarrow 0} \frac{f(x+tw)-f(x)}{t}\geq \langle \upsilon, w\rangle \text{ for all } w\in\R^n\right\};
\]
it holds that $\partial h(\bar x)\subseteq \partial^\circ h(\bar x)$; see \cite[Theorem~5.2.22]{BorZhu04}. Finally, for a proper closed function $f$, we say that $\bar x$ is a stationary point of $f$ if $0\in \partial f(\bar x)$. All local minimizers are stationary points according to \cite[Theorem~10.1]{RockWets98}.

For a nonempty closed set $C$, we define the indicator function as
\[
\delta_C(x):=\begin{cases}
0 & x\in C,\\
\infty & x\notin C.
\end{cases}
\]
The normal cone (resp., regular normal cone) of $ C $ at an $ x\in C $ is given by $ N_C(x):=\partial \delta_C(x) $ (resp., $\widehat{N}_C(x) := \widehat{\partial}\delta_C(x)$). The distance from a point $ x $ to $ C $ is denoted by $ {\rm dist}(x, C) $. The set of points in $C$ that are closest to $x$ is denoted by ${\rm Proj}_C(x)$. The convex hull of $C$ is denoted by ${\rm conv}\,C$.

We next recall the Kurdyka-{\L}ojasiewicz (KL) property. This property and the associated notion of KL exponent has been used extensively in the convergence analysis of various first-order methods; see \cite{Attouch09, Attouch10, Attouch13, BolSabTeb14, LiPong16}.
\begin{definition}[Kurdyka-{\L}ojasiewicz property and exponent]
	We say that a proper closed function $ h:\R^n\rightarrow (-\infty, \infty] $ satisfies the Kurdyka-{\L}ojasiewicz (KL) property at an $ \widehat x \in {\rm dom}\,\partial h $ if there are $ a\in(0,\infty] $, a neighborhood $ U $ of $ \widehat x $ and a continuous concave function $ \varphi:[0,a)\rightarrow [0,\infty) $ with $ \varphi(0)=0 $ such that
\begin{enumerate}[{\rm (i)}]
	\item $ \varphi $ is continuously differentiable on $ (0,a) $ with  $ \varphi^\prime>0 $ on $ (0, a) $;
	\item for every $ x\in U $ with $ h(\widehat x)< h(x) < h(\widehat x)+a $, it holds that
	\begin{equation}\label{KL}
		\varphi^\prime(h(x)-h(\widehat x)){\rm dist}(0, \partial h(x)) \geq 1.
	\end{equation}
\end{enumerate}
If $ h $ satisfies the KL property at $ \widehat x\in{\rm dom}\,\partial h $ and the $ \varphi $ in \eqref{KL} can be chosen as $ \varphi(\nu) = a_0\nu^{1-\theta} $ for some $a_0 > 0$ and $ \theta\in[0,1) $, then we say that $ h $ satisfies the KL property at $ \widehat x $ with exponent $ \theta $.

A proper closed function $ h $ satisfying the KL property at every point in $ {\rm dom}\,\partial h $ is called a KL function, and a proper closed function $ h $ satisfying the KL property with exponent $ \theta\in [0, 1) $ at every point in $  {\rm dom}\,\partial h $ is called a KL function with exponent $ \theta $.
\end{definition}

KL functions arise naturally in various applications. \add{For instance, proper closed semi-algebraic functions are KL functions with some exponent $ \theta\in[0,1) $;} see \cite{Attouch10} and \cite[Theorem~3.1]{BolDan07}, respectively.

Another notion that will be needed in our discussion later is (subdifferential) regularity; see \cite[Definition~6.4]{RockWets98} and \cite[Definition~7.25]{RockWets98}.
\begin{definition}
	A nonempty closed set $C$ is regular at $x\in C$ if $N_C(x) = \widehat N_C(x)$, and a proper closed function $h$ is (subdifferentially) regular at $x\in {\rm dom}\,h$ if its epigraph ${\rm epi}\,h:=\{(x,t)\in \R^n\times \R:\; h(x)\le t\}$ is regular at $(x,h(x))$.
\end{definition}

According to \cite[Example~7.28]{RockWets98}, continuously differentiable functions are regular everywhere. Thus, the constraint functions in \eqref{prob:LS} and \eqref{prob:Lorentzian} are regular everywhere. In addition, a nonsmooth regular function particularly relevant to our discussion is the objective function of \eqref{Pold}. Indeed, in view of \cite[Corollary~1.111(i)]{Morduk06}, it holds that:
\begin{equation}\label{subdatbarx}
  \text{At any $\bar x \neq 0$, $\frac{\|\cdot\|_1}{\|\cdot\|} $ is regular and $\partial \frac{\|\bar x\|_1}{\|\bar x\|} = \frac{1}{\|\bar x\|}\partial\| \bar x \|_1- \frac{\|\bar x\|_1}{\|\bar x\|^3}\bar x$}.
\end{equation}
We will also need the following auxiliary lemma concerning the subdifferential of a particular class of functions in our analysis in section \ref{sec6}.
\begin{lemma}\label{lem:regularity2barq}
	Let $ q = P_1 - P_2 $ with $ P_1: \R^n\rightarrow \R $ being continuously differentiable and
	$ P_2:\R^n\rightarrow \R  $ being convex continuous. Then for any $x\in \R^n$, we have
\begin{equation}\label{clarkesubdiffq}
\partial^\circ q(x) = \nabla P_1(x)-\partial P_2(x).
\end{equation}
\end{lemma}
\begin{proof}
Note that for any $x\in \R^n$, we have
\[
\partial^\circ q(x) \overset{\rm (a)}= \nabla P_1(x) + \partial^\circ (- P_2)(x) \overset{\rm (b)}= \nabla P_1(x) - \partial^\circ P_2(x) \overset{\rm (c)}= \nabla P_1(x) - \partial P_2(x),
\]
where (a) follows from Corollary 1 of \cite[Proposition~2.3.3]{Clarke90}, (b) holds because of \cite[Proposition~2.3.1]{Clarke90} and (c) follows from \cite[Proposition~2.2.7]{Clarke90}. 
\end{proof}

\section{Solution existence of model \eqref{Pold}}\label{sec3}

In this section, we establish the existence of optimal solutions to problem \eqref{Pold} under suitable assumptions. A similar discussion was made in \cite[Theorem~2.2]{RaWaDoLo18}, where the existence of \emph{local minimizers} was established under the strong null space property (see \cite[Definition~2.1]{RaWaDoLo18}) of the sensing matrix $ A $. It was indeed shown that any sufficiently sparse solution of $ Ax = b $ is a local minimizer for problem~\eqref{Pold}, under the strong null space property. Here, our discussion focuses on the existence of {\em globally optimal solutions}, and our analysis is based on the spherical section property (SSP) \cite{Vavasis09,Zhang13}.
\begin{definition}[Spherical section property \cite{Vavasis09,Zhang13}]
	Let $ m $, $ n $ be two positive integers such that $ m<n $. Let $ V $ be an $ (n-m) $-dimensional subspace of $ \R^n $ \add{and $ s $ be a positive integer}. We say that $ V $ has the $ s $-spherical section property if
	\[
	\inf_{v\in V\backslash\{0\}}\frac{\|v\|_1}{\|v\|}\geq \sqrt{\frac{m}{s}}.
	\]
\end{definition}
\begin{remark}
	According to \cite[Theorem~3.1]{Zhang13}, if $A\in \R^{m\times n}$ ($m < n$) is a random matrix with i.i.d. standard Gaussian entries, then its $ (n-m) $-dimensional nullspace has the $ s $-spherical section property for $ s=c_1({\rm log}(n/m)+1) $ with probability at least $ 1-e^{-c_0(n-m)} $, where $c_0 $ and $c_1$ are positive constants independent of $m$ and $n$.
\end{remark}

We now present our analysis.
We first characterize the existence of unbounded minimizing sequences of \eqref{Pold}: recall that $ \{x^t\} $ is called a minimizing sequence of \eqref{Pold} if $ Ax^t=b $ for all $ t $ and
$\lim_{t\rightarrow \infty} \frac{\|x^t\|_1}{\|x^t\|} = \nu^*_{cs}$.
Our characterization is related to the following auxiliary problem, where $A$ is as in \eqref{Pold}:
\begin{equation}\label{limit_prob}
	\nu^*_d:= \inf\left\{\frac{\|d\|_1}{\|d\|}:\; Ad = 0, d\neq 0 \right\}.
\end{equation}
\begin{lemma}\label{lem:minseq}
	Consider \eqref{Pold} and \eqref{limit_prob}. Then $ \nu_{cs}^* =\nu_d^*$ if and only if there exists a minimizing sequence of \eqref{Pold} that is unbounded.
\end{lemma}
\begin{proof}
    We first suppose that there exists an unbounded minimizing sequence $ \{x^t\} $ of \eqref{Pold}. By passing to a subsequence if necessary, we may assume without loss of generality that $ \|x^t\|\rightarrow \infty$ and that $\lim_{t\rightarrow\infty}\frac{x^t}{\|x^t\|} = x^*$ for some $x^*$ with $\|x^*\|=1$. Then we have $ \|x^*\|_1 =  \nu^*_{cs} $ using the definition of minimizing sequence, and
	\begin{equation}\label{Ad=0}
		Ax^*= \lim_{t\rightarrow \infty} \frac{Ax^t}{\|x^t\|} = \lim_{t\rightarrow \infty} \frac{b}{\|x^t\|}=0.
	\end{equation}
	One can then see that
	\begin{equation}\label{haha2}
	\nu_{d}^*\leq \frac{\|x^*\|_1}{\|x^*\|} = \|x^*\|_1= \nu_{cs}^* < \infty.
	\end{equation}
    Next, fix any $ x $ such that $ Ax = b $ and choose any $ d\neq 0 $ satisfying $ Ad=0 $ (these exist thanks to $\nu_{d}^*\le \nu_{cs}^*<\infty$). Then it holds that
	\[
	\nu_{cs}^* \leq \frac{\|x+sd\|_1}{\|x+sd\|}
	\]
	for any $ s\in\R $. It follows from the above display that
	\[
	\nu_{cs}^* \leq \lim_{s\rightarrow \infty} \frac{\|x+sd\|_1}{\|x+sd\|} = \frac{\|d\|_1}{\|d\|}.
	\]
	Then we have $ \nu_{cs}^*\leq\nu_d^*$ by the arbitrariness of $ d $. This together with \eqref{haha2} shows that $ \nu_{cs}^*=\nu_d^*$.
	
	We next suppose that $ \nu^*_{cs} =  \nu^*_d $. Since $\nu^*_{cs} < \infty$ (thanks to $A^{-1}\{b\}\neq\emptyset$), there exists a sequence $ \{d^k\} $ satisfying $Ad^k = 0$ and $d^k\neq 0$ such that $	\lim_{k\rightarrow\infty}\frac{\|d^k\|_1}{\|d^k\|} = \nu^*_d$. Passing to a further subsequence if necessary, we may assume without loss of generality that $ \lim_{k\to \infty}\frac{d^k}{\|d^k\|} = d^* $ for some $d^*$ with $ \|d^*\|=1$. It then follows that
\[
Ad^* = \lim_{k\to \infty}\frac{Ad^k}{\|d^k\|} = 0 \ \ {\rm and}\ \ \|d^*\|_1 = \lim_{k\to\infty}\left\|\frac{d^k}{\|d^k\|}\right\|_1 = \nu^*_d.
\]
Now, choose any $ x^0 $ such that $ Ax^0=b $ and define $ x^t = x^0+td^* $ for each $t =1,2,\ldots$ Then we have $ Ax^t = b $ for all $ t $. Moreover $ \|x^t\|\rightarrow \infty $ as $ t\rightarrow \infty $ and
	\[
	\lim_{t\rightarrow\infty}\frac{\|x^t\|_1}{\|x^t\|} =\frac{\|d^*\|_1}{\|d^*\|}= \nu^*_d = \nu^*_{cs}.
	\]
	Thus, $\{x^t\}$ is an unbounded minimizing sequence for \eqref{Pold}. This completes the proof.
\end{proof}

We are now ready to present the theorem on solution existence for \eqref{Pold}.
\begin{theorem}[Solution existence for \eqref{Pold}]\label{thm:solexts_Pold}
	Consider \eqref{Pold}. Suppose that $\ker A $ has the $ s $-spherical section property for some $s > 0$ and there exists $ \widetilde x\in\R^n $ such that $ \|\widetilde x\|_0< m/s $ and $ A\widetilde x=b $. Then the optimal value $ \nu^*_{cs} $ of \eqref{Pold} is attainable, i.e., the set of optimal solutions of \eqref{Pold} is nonempty.
\end{theorem}
\begin{proof}
	According to the $ s $-spherical property of $ \ker A $ and the definition of $ \nu^*_d $ in \eqref{limit_prob}, we see that $ \nu^*_d \geq \sqrt{\frac{m}{s}} $. It then follows that
	\[
	\nu^*_{cs}\overset{\rm (a)}\leq \frac{\|\widetilde{x}\|_1}{\|\widetilde{x}\|} \overset{\rm (b)}\leq \sqrt{\|\widetilde{x}\|_0} \overset{\rm (c)}<  \sqrt{\frac{m}{s}} \leq \nu^*_d,
	\]
    where (a) follows from the definition of $\nu^*_{cs}$ and the fact that $A\widetilde x = b$, (b) follows from Cauchy-Schwarz inequality and (c) holds by our assumption.
	Invoking Lemma~\ref{lem:minseq} and noting $\nu^*_{cs} < \infty$, we see that there is a bounded minimizing sequence $ \{x^t\} $ for \eqref{Pold}. We can then pass to a convergent subsequence $ \{x^{t_j}\} $ so that $ \lim_{j\rightarrow \infty} x^{t_j}=x^* $ for some $ x^* $ satisfying $ Ax^*=b $. Since $b \neq 0$, this means in particular that $x^* \neq 0$. We then have upon using the continuity of $ \frac{\|\cdot\|_1}{\|\cdot\|} $ at $x^*$ and the definition of minimizing sequence that
	\[
	\frac{\|x^*\|_1}{\|x^*\|} = \lim_{j\rightarrow \infty}\frac{\|x^{t_j}\|_1}{\|x^{t_j}\|} = \nu_{cs}^*.
	\]
This shows that $x^*$ is an optimal solution of \eqref{Pold}. This completes the proof.
\end{proof}

\section{KL exponent of $F$ in \eqref{Fnoiseless} and global convergence of Algorithm~\ref{alg:lou}}\label{sec4}

In this section, we discuss the KL exponent of \eqref{Fnoiseless} and its implication on the convergence rate of the algorithm proposed in \cite[Eq.~11]{WangYanLou19} for solving \eqref{Pold}.
For ease of reference, this algorithm is presented as Algorithm~\ref{alg:lou} below. It was shown in \cite{WangYanLou19} that if the sequence $\{x^t\}$ generated by this algorithm is bounded, then any accumulation point is a stationary point of $F$ in \eqref{Fnoiseless}.
\begin{algorithm}
	\caption{The algorithm proposed in \cite[Eq.~11]{WangYanLou19} for \eqref{Pold}}\label{alg:lou}
	\begin{algorithmic}
		\STATE {\bf Step 0.} Choose $x^0$ with $Ax^0 = b$ and $ \alpha>0 $. Set $ \omega_0 = \|x^0\|_1/\|x^0\| $ and $t=0$.
		
		{\bf Step 1.} Solve the subproblem
		\begin{equation}\label{subp_lou}
			\begin{array}{rl}
				x^{t+1} = \displaystyle\argmin_{x\in \R^{n}} & \displaystyle \|x\|_1 - \frac{\omega_t}{\|x^t\|}\langle x,x^t\rangle		 + \frac{\alpha}2\|x - x^t\|^2\\
				{\rm s.t.}\quad& Ax = b.
			\end{array}
		\end{equation}
		
		{\bf Step 2.} Compute $\omega_{t+1} = \|x^{t+1}\|_1/\|x^{t+1}\|$. Update $t\leftarrow t+1$ and go to {\bf Step 1.}
	\end{algorithmic}
\end{algorithm}

Here, we \add{first remark} that if the sequence $\{x^t\}$ generated by Algorithm~\ref{alg:lou} is bounded, then it converges to a stationary point $x^*$ of $F$ in \eqref{Fnoiseless}. The argument is standard \add{(see \mbox{\cite{Attouch10, Attouch13, BolSabTeb14}})}, making use of {\bf H1}, \textbf{H2}, \textbf{H3} in \cite[Section~2.3]{Attouch13}. We include the proof for the ease of readers.
\begin{proposition}[Global convergence of Algorithm~\ref{alg:lou}]
	Consider \eqref{Pold}. Let $ \{x^t\} $ be the sequence generated by Algorithm~\ref{alg:lou} and suppose that $ \{x^t\}  $ is bounded. Then $ \{x^t\}  $ converges to a stationary point of $F$ in \eqref{Fnoiseless}.
\end{proposition}
\begin{proof}
First, according to \cite[Lemma~1]{WangYanLou19}, the sequence $ \{\omega_t\} $ generated by Algorithm~\ref{alg:lou} enjoys the following sufficient descent property:
\begin{equation}\label{wdecrease_lou}
	\omega_t-\omega_{t+1}\geq \frac{\alpha}{2\|x^{t+1}\|}\|x^{t+1}-x^t\|^2.
\end{equation}
Now, if we let $ \lambda^t $ denote a Lagrange multiplier of the subproblem \eqref{subp_lou} at iteration $ t $, one then see from the first-order optimality condition that
\begin{equation}\label{foc_lou}
-A^T\lambda^t+\frac{\|x^t\|_1}{\|x^t\|^2}x^t-\alpha(x^{t+1}-x^t)\in\partial \|x^{t+1}\|_1.
\end{equation}
On the other hand, using \eqref{subdatbarx} and noting that $x^t\neq 0$ for all $t$, we have
\[
\begin{aligned}
&\frac{1}{\|x^{t+1}\|}\partial \|x^{t+1}\|_1 -\frac{\|x^{t+1}\|_1}{\|x^{t+1}\|^3}x^{t+1} + \frac{A^T\lambda^t}{\|x^{t+1}\|}\\
& = \partial\frac{\|x^{t+1}\|_1}{\|x^{t+1}\|} + \frac{A^T\lambda^t}{\|x^{t+1}\|} \subset \partial\frac{\|x^{t+1}\|_1}{\|x^{t+1}\|} + N_{A^{-1}\{b\}}(x^{t+1}) = \partial F(x^{t+1}),
\end{aligned}
\]
where the last equality follows from \cite[Corollary~10.9]{RockWets98}, the regularity at $x^{t+1}$ of $\frac{\|\cdot\|_1}{\|\cdot\|}$ (see \eqref{subdatbarx}) and $\delta_{A^{-1}\{b\}}(\cdot)$ (see \cite[Theorem~6.9]{RockWets98}), and the definition of $ F $ in \eqref{Fnoiseless}.
Combining \eqref{foc_lou} and the above display, we obtain that
\[
\frac{1}{\|x^{t+1}\|}\left(\frac{\|x^t\|_1}{\|x^t\|^2}x^t-\frac{\|x^{t+1}\|_1}{\|x^{t+1}\|^2}x^{t+1}\right) -\frac{\alpha}{\|x^{t+1}\|}(x^{t+1}-x^t)\in\partial F(x^{t+1}).
\]
On the other hand, since $\|x^t\|\ge \inf_{y\in A^{-1}\{b\}}\|y\| > 0$ for all $t$ (thanks to $Ax^t = b$ and $b\neq 0$) and $ \{x^t\} $ is bounded, we see that there exists $C_0 > 0$ so that
\[
\left\|\frac{\|x^t\|_1}{\|x^t\|^2}x^t-\frac{\|x^{t+1}\|_1}{\|x^{t+1}\|^2}x^{t+1}\right\|\le C_0\|x^{t+1} - x^t\| \text{ for all } t.
\]
Thus, in view of the above two displays, we conclude that
\[
{\rm dist}(0, \partial F(x^{t+1})) \leq \frac{C_0+\alpha}{\inf_{y\in A^{-1}\{b\}}\|y\|}\|x^{t+1}-x^t\| \text{ for all } t.
\]
Using the boundedness of $\{x^t\}$, \eqref{wdecrease_lou}, the above display and the continuity of $F$ on its domain, we see that the conditions {\bf H1}, \textbf{H2}, \textbf{H3} in \cite[Section~2.3]{Attouch13} are satisfied. Since $F$ is clearly proper closed semi-algebraic and hence a KL function, we can then invoke \cite[Theorem~2.9]{Attouch13} to conclude that $ \{x^t\} $ converges to a stationary point of $F$.
\end{proof}

While it is routine to show that the sequence $\{x^t\}$ generated by Algorithm~\ref{alg:lou} is convergent when it is bounded, it is more challenging to deduce the asymptotic convergence rate: the latter typically requires an estimate of the KL exponent of $F$ in \eqref{Fnoiseless}, which was used in the above analysis. In what follows, we will show that the KL exponent of $F$ is $\frac12$. To do this, we will first establish a calculus rule for deducing the KL exponent of a fractional objective from the difference between the numerator and (a suitable scaling of) the denominator: this is along the line of the calculus rules for KL exponents developed in \cite{LiPong16,YuLiPong19,LiuPong19}, and can be of independent interest.

\subsection{KL exponent of fractional functions}
Let $f:\R^n\to \R\cup\{\infty\}$ be proper closed and $g:\R^n\to \R$ be a continuous nonnegative function that is continuously differentiable on an open set containing ${\rm dom}\,f$. Suppose that $\inf f\ge 0$ and $\inf_{{\rm dom}\,f} g > 0$. We consider the following fractional programming problem:
\begin{equation}\label{4prop:fractional}
	\min_{x} G(x) := \frac{f(x)}{g(x)}.
\end{equation}
In algorithmic developments for solving \eqref{4prop:fractional} (see, for example, \cite{CrFeNg08,Din67}), it is customary to consider functions of the following form
\begin{equation}\label{4Hu}
	H_u(x) := f(x) - \frac{f(u)}{g(u)}g(x),
\end{equation}
where $u$ typically carries information from the previous iterate. \add{In the literature, KL-type assumptions are usually imposed on $ G $ or $ H_u $ for establishing the global convergence of the sequence generated by first-order methods for solving \eqref{4prop:fractional}; see, for example, the discussions in {\cite[Theorem~16]{BotCse17} and \cite[Theorem~5.5]{BotLi20}}.} \add{Here,} we study a relationship between the KL exponent of $G$ in \eqref{4prop:fractional} and that of $H_{\bar x}$ in \eqref{4Hu} when $\bar x$ is a stationary point of $G$.

\begin{theorem}[KL exponent of fractional functions]\label{4fractionalthm}
	Let $f:\R^n\to \R\cup\{\infty\}$ be a proper closed function with $\inf f\ge 0$ and $g:\R^n\to \R$ be a continuous nonnegative function that is continuously differentiable on an open set containing ${\rm dom}\,f$ with $\inf_{{\rm dom}\,f} g > 0$. Assume that one of the following conditions hold:
	\begin{enumerate}[{\rm (i)}]
		\item $f$ is locally Lipschitz.
		\item $f = h + \delta_D$ for some continuously differentiable function $h$ and nonempty closed set $D$.
		\item $f = h + \delta_D$ for some locally Lipschitz function $h$ and nonempty closed set $D$, and $h$ and $D$ are regular at every point in $D$.
	\end{enumerate}
	Let $\bar x$ be such that $0 \in \partial G(\bar x)$, where $G$ is defined as in \eqref{4prop:fractional}. Then $\bar x\in {\rm dom}\,\partial H_{\bar x}$. If $H_{\bar x}$ defined as in \eqref{4Hu} satisfies the KL property with exponent $\theta\in [0,1)$ at $\bar x$, then so does $G$.
\end{theorem}

\begin{proof}
	It is clear that ${\rm dom}\, H_{\bar x} = {\rm dom}\,f = {\rm dom}\,G$.
	We first argue that under the assumptions on $f$ and $ g $, we have for any $x\in {\rm dom}\,G$ that
	\begin{equation}\label{regularity}
		\partial H_{\bar x}(x) = \partial f(x) - G(\bar x)\nabla g(x)\ \ {\rm and}\ \ \partial G(x) = \frac1{g(x)}\left(\partial f(x) - G(x)\nabla g(x)\right).
	\end{equation}
	Indeed, in all cases, the first relation in \eqref{regularity} follows from \cite[Exercise~8.8(c)]{RockWets98}. When $f$ is locally Lipschitz, the second relation in \eqref{regularity} follows from \cite[Corollary~1.111(i)]{Morduk06}. When $f = h + \delta_D$ for some continuously differentiable function $h$ and nonempty closed set $D$, the second relation in \eqref{regularity} follows by first applying \cite[Exercise~8.8(c)]{RockWets98} to $G = \frac{h}g + \delta_D$, then applying the usual quotient rule to the differentiable function $\frac{h}g$, and subsequently using $\partial f = \nabla h + \partial \delta_D$ (thanks to \cite[Exercise~8.8(c)]{RockWets98}). Finally, when $f = h + \delta_D$ for some locally Lipschitz function $h$ and nonempty closed set $D$ with $h$ and $D$ being  regular at every point in $D$, we have that the function $\frac{h}g$ is  regular for all $x\in D$ in view of \cite[Corollary~1.111(i)]{Morduk06}. This together with the regularity of $D$ gives
	\[
	\begin{aligned}
		\partial G(x) &= \partial\left(\frac{h}{g}\right)(x) + \partial\delta_D(x)\\
		& = \frac{g(x)\partial h(x) - h(x)\nabla g(x)}{g(x)^2} + \partial\delta_D(x)\\
		& = \frac{g(x)\partial f(x) - f(x)\nabla g(x)}{g(x)^2},
	\end{aligned}
	\]
	where the first and the last equalities follow from \cite[Corollary~10.9]{RockWets98} and \cite[Exercise~8.14]{RockWets98}, and the second equality follows from \cite[Corollary~1.111(i)]{Morduk06}.
	
	Now, in view of \eqref{regularity}, we have ${\rm dom}\, \partial H_{\bar x} = {\rm dom}\,\partial f = {\rm dom}\,\partial G$. In addition, in all three cases, it holds that ${\rm dom}\, f = {\rm dom}\,\partial f$. Indeed, when $f$ is locally Lipschitz, this claim follows from Exercise~8(c) of \cite[Section~6.4]{BorLew06}. When $f = h+\delta_D$ as in (ii), the claim follows from \cite[Exercise~8.8(c)]{RockWets98}, while for case (iii), we have ${\rm dom}\, f = {\rm dom}\,\partial f = D$ in view of \cite[Corollary~10.9]{RockWets98}, \cite[Exercise~8.14]{RockWets98} and Exercise~8(c) of \cite[Section~6.4]{BorLew06}. Consequently, in all three cases, we have
	\[
	\Xi:= {\rm dom}\,G = {\rm dom}\,\partial G = {\rm dom}\,H_{\bar x}= {\rm dom}\,\partial H_{\bar x} = {\rm dom}\, f = {\rm dom}\,\partial f,
	\]
	and $H_{\bar x}$ is continuous relative to $\Xi$. In particular, $\bar x\in {\rm dom}\,\partial G= {\rm dom}\,\partial H_{\bar x}$.
	
	Let $U$ be the open set containing ${\rm dom}\,f$ on which $g$ is continuously differentiable.
	Since $H_{\bar x}$ satisfies the KL property with exponent $\theta$ at $\bar x$ and is continuous \add{relative to} $\Xi$, there exist $\epsilon > 0$ and $c > 0$ so that $B(\bar x,2\epsilon)\subseteq U$ and
	\begin{equation}\label{4KLH}
		{\rm dist}(0,\partial H_{\bar x}(x)) \ge c (H_{\bar x}(x) - H_{\bar x}(\bar x))^\theta = c (H_{\bar x}(x))^\theta
	\end{equation}
	whenever $x\in \Xi$, $H_{\bar x}(x) > 0$ and $\|x - \bar x\|\le \epsilon$. Let $M\!\! := \sup_{\|x-\bar x\|\le\epsilon}\max\{g(x),\|\nabla g(x)\|\}$, which is finite as $g$ is continuously differentiable on $U\supseteq B(\bar x,2\epsilon)$. \add{Using the facts that $\theta\in [0,1)$, $ H_{\bar x} $ is continuous relative to $\Xi$,} $H_{\bar x}(\bar x) = 0$ and $\inf_{{\rm dom}\,f} g > 0$, \add{we deduce that} there exists $\epsilon'\in (0,\epsilon)$ such that
	\begin{equation}\label{4Hineq}
		|H_{\bar x}(x)|^{1-\theta} \le \frac{c\inf_{{\rm dom}\,f} g}{2M}\ \ \ \mbox{ whenever }\|x - \bar x\|\le \epsilon'\ {\rm and}\ x\in \Xi,
	\end{equation}
	where $c$ is given in \eqref{4KLH}.
	
	Now, consider any $x\in \Xi$ satisfying $\|x - \bar x\|\le \epsilon'$ and $G(\bar x) < G(x) < G(\bar x)+\epsilon'$. Then we have from \eqref{regularity} that
	\begin{align*}
	&{\rm dist}(0,\partial G(x))  = \frac{1}{g(x)}\inf_{\xi\in \partial f(x)}\|\xi - G(x)\nabla g(x)\|\overset{\rm (a)}\ge \frac{1}{M}\inf_{\xi\in \partial f(x)}\|\xi - G(x)\nabla g(x)\|\\
	&  \overset{\rm (b)}\ge \frac{1}{M}\inf_{\xi\in \partial f(x)}\|\xi - G(\bar x)\nabla g(x)\| - \frac1M |G(x) - G(\bar x)|\|\nabla g(x)\|\\
	&  \overset{\rm (c)}\ge  \frac{1}{M}\inf_{\xi\in \partial f(x)}\|\xi - G(\bar x)\nabla g(x)\| - (G(x) - G(\bar x))\\
	&  = \frac{1}{M}{\rm dist}(0,\partial H_{\bar x}(x)) - \frac1{g(x)}H_{\bar x}(x) \overset{\rm (d)}\ge \frac{1}{M}{\rm dist}(0,\partial H_{\bar x}(x))  - \frac1{\inf_{{\rm dom}\,f} g}H_{\bar x}(x)\\
	& \overset{\rm (e)}\ge \frac{c}{M}(H_{\bar x}(x))^\theta - \frac1{\inf_{{\rm dom}\,f} g}H_{\bar x}(x)\overset{\rm (f)}\ge \frac{c}{2M}(H_{\bar x}(x))^\theta \\
& = \frac{c(g(x))^\theta}{2M}(G(x) - G(\bar x))^\theta \overset{\rm (g)}\ge \frac{c(\inf_{{\rm dom}\,f} g)^\theta}{2M}(G(x) - G(\bar x))^\theta,
	\end{align*}
	where (a) holds because $g(x)\le M$, (b) follows from the triangle inequality, (c) holds because $\|\nabla g(x)\|\le M$ and $G(x) > G(\bar x)$, (d) holds because $H_{\bar x}(x) > 0$ (thanks to $G(x) > G(\bar x)$), (e) then follows from \eqref{4KLH} and (f) follows from \eqref{4Hineq} and the fact that $H_{\bar x}(x) > 0$. Finally, (g) holds because $G(x) > G(\bar x)$. This completes the proof.
\end{proof}

\subsection{KL exponent of $ F $ in \eqref{Fnoiseless}}\label{sec42}

Before proving our main result concerning the KL exponent of $F$ in \eqref{Fnoiseless}, we also need the following simple proposition.

\begin{proposition}\label{chainrule}
	Let $p$ be a proper closed function, and let $\bar x\in {\rm dom}\,p$ be such that $p(\bar x)>0$. Then the following statements hold.
	\begin{enumerate}[{\rm (i)}]
		\item We have $\partial (p^2)(x) = 2p(x) \partial p(x)$ for all $x$ sufficiently close to $\bar x$.
		\item Suppose in addition that $\bar x\in {\rm dom}\,\partial (p^2)$ and $p^2$ satisfies the KL property at $\bar x$ with exponent $\theta\in [0,1)$. Then $p$ satisfies the KL property at $\bar x$ with exponent $\theta\in [0,1)$.
	\end{enumerate}
\end{proposition} 
\begin{proof}
	Since $p(\bar x) > 0$ and $p$ is closed, there exists $\epsilon > 0$ so that
	\begin{equation*}
		0< p(x) < \infty
	\end{equation*}
	whenever $\|x - \bar x\|\le \epsilon$ and $x\in {\rm dom}\,p$. Then we deduce from \cite[Lemma~1]{NghThe09} that
	\begin{equation}\label{haha1}
		\widehat \partial (p^2)(x) = 2p(x)\widehat \partial p(x)\ \ {\rm whenever}\ x\in {\rm dom}\,p\ {\rm and}\ \|x - \bar x\|\le \epsilon.
	\end{equation}

	Using \eqref{haha1}, and invoking the definition of limiting subdifferential and by shrinking $\epsilon$ if necessary, we deduce that
	\begin{equation}\label{squarerule}
		\partial (p^2)(x) = 2p(x) \partial p(x) \ \ {\rm whenever}\ x\in {\rm dom}\,p\ {\rm and}\ \|x - \bar x\|\le \epsilon.
	\end{equation}
	In particular, if $\bar x\in {\rm dom}\,\partial (p^2)$, then $\bar x\in {\rm dom}\,\partial p$.
	
	When $p^2$ also satisfies the KL property at $\bar x$ with exponent $\theta$, by shrinking $\epsilon$ further if necessary, we see that there exists $c > 0$ so that
	\begin{equation}\label{hahaKL}
		{\rm dist}(0,\partial(p^2)(x)) \ge c(p^2(x) - p^2(\bar x))^\theta,
	\end{equation}
	whenever $p^2(\bar x) < p^2(x) < p^2(\bar x) + \epsilon(2p(\bar x) + \epsilon)$ and $\|x - \bar x\|\le \epsilon$. Thus, for $x\in {\rm dom}\,\partial p$ satisfying $\|x - \bar x\|\le \epsilon$ and $p(\bar x) < p(x) < p(\bar x) + \epsilon$, we have from \eqref{squarerule} that
	\[
	\begin{aligned}
	&{\rm dist}(0,\partial p(x)) = \frac1{2p(x)}{\rm dist}(0,\partial(p^2)(x))\ge \frac1{2p(\bar x)+2\epsilon}{\rm dist}(0,\partial(p^2)(x))\\
	& \overset{\rm (a)}\ge \frac{c}{2p(\bar x)+2\epsilon}(p^2(x) - p^2(\bar x))^\theta= \frac{c}{2p(\bar x)+2\epsilon}(p(x) + p(\bar x))^\theta(p(x) - p(\bar x))^\theta\\
	& \ge \frac{c[p(\bar x)]^{\theta}}{2^{1-\theta}(p(\bar x) + \epsilon)}(p(x) - p(\bar x))^\theta,
	\end{aligned}
	\]
	where (a) follows from \eqref{hahaKL}.
	This completes the proof.
\end{proof}

We are now ready to show that the KL exponent of $F$ in \eqref{Fnoiseless} is $\frac12$. We remark that if the set ${\cal X} := \{x:\; 0\in \partial F(x)\}$ is empty, then this claim holds trivially in view of \cite[Lemma~2.1]{LiPong16}. However, in general, one can have ${\cal X}\neq \emptyset$. Indeed, according to Theorem~\ref{thm:solexts_Pold} and \cite[Theorem~10.1]{RockWets98}, we have ${\cal X}\neq \emptyset$ with high probability when $A$ is generated in a certain way.
\begin{theorem}
	The function $ F $ in \eqref{Fnoiseless}
	is a KL function with exponent $\frac12$.
\end{theorem}
\begin{proof}
	In view of \cite[Lemma~2.1]{LiPong16}, it suffices to look at the KL exponent at a stationary point $\bar x$ of $F$.
	For any $\bar x$ satisfying $0 \in\partial F(\bar x)$, we have $F(\bar x) > 0$ since $b\neq 0$. Moreover, we have $0\in \partial (F^2)(\bar x)$ in view of Proposition~\ref{chainrule}(i). Next, note that the function
	\[
	 F_1(x) := \|x\|_1^2 - \frac{\|\bar x\|^2_1}{\|\bar x\|^2}\|x\|^2 + \delta_{A^{-1}\{b\}}(x)
	\]
	can be written as $\min_{\sigma \in {\frak R}}\{Q_\sigma(x) + P_\sigma(x)\}$, where ${\frak R} = \{u\in \R^n:\; u_i \in \{1,-1\}\ \ \forall i\}$, and $Q_\sigma$ are quadratic functions (nonconvex) and $P_\sigma$ are polyhedral functions indexed by $\sigma$: indeed, for each $\sigma\in {\frak R}$, one can define $P_\sigma$ as the indicator function of the set $\{x:\; Ax = b, \sigma\circ x\ge 0\}$, where $\circ$ denotes the Hadamard product, and $Q_\sigma(x) := (\langle\sigma,x\rangle)^2-\frac{\|\bar x\|^2_1}{\|\bar x\|^2}\|x\|^2$. Then, in view of \cite[Corollary~5.2]{LiPong16}, $F_1$ is a KL function with exponent $\frac12$. Since the convex function $\|\cdot\|_1^2$ is regular everywhere and the convex set $A^{-1}\{b\}$ is regular at every $x\in A^{-1}\{b\}$ (thanks to \cite[Theorem~6.9]{RockWets98}), we deduce using Theorem~\ref{4fractionalthm} that the function
	\[
	x\mapsto \frac{\|x\|^2_1}{\|x\|^2} + \delta_{A^{-1}\{b\}}(x)
	\]
	satisfies the KL property at $\bar x$ with exponent $\frac12$. The desired conclusion now follows from this and Proposition~\ref{chainrule}(ii).
\end{proof}

Equipped with the result above, by following the line of arguments in \cite[Theorem~2]{Attouch09}, one can conclude further that the sequence $\{x^t\}$ generated by Algorithm~\ref{alg:lou} converges locally linearly to a stationary point of $F$ in \eqref{Fnoiseless} if the sequence is bounded. The proof is standard and we omit it here for brevity.
\begin{theorem}[Convergence rate of Algorithm~\ref{alg:lou}]
	Consider \eqref{Pold}. Let $ \{x^t\} $ be the sequence generated by Algorithm~\ref{alg:lou} and suppose that $ \{x^t\}  $ is bounded. Then $ \{x^t\}  $ converges to a stationary point $ x^* $ of $F$ in \eqref{Fnoiseless} and there exist $ \underline{t} \in \mathbb{N}_+$, $ a_0\in(0,1) $ and $ a_1>0 $ such that
	\[
	\|x^t-x ^*\|\leq a_1a_0^t \quad \text{ whenever }\, t>\underline{t}.
	\]
\end{theorem}

\section{Compressed sensing with noise based on $\ell_1/\ell_2$ minimization}\label{sec:noisyCS}
In the previous sections, we have been focusing on the model \eqref{Pold}, which corresponds to noiseless compressed sensing problems. In this section and the next, we will be looking at \eqref{P0}.
We will discuss conditions for existence of solutions and derive some first-order optimality conditions for \eqref{P0} in this section. An algorithm for solving \eqref{P0} will be proposed in the next section and will be shown to generate sequences that cluster at ``critical" points in the sense defined in this section, under suitable assumptions.

\subsection{Solution existence}
Clearly, if $q$ in \eqref{P0} is in addition level-bounded, then the feasible set is compact and hence the set of optimal solutions is nonempty. However, in applications such as \eqref{prob:LS}, \eqref{prob:Lorentzian} and \eqref{prob:rcs}, the corresponding $q$ is not level-bounded.
Here, we discuss solution existence for \eqref{prob:LS} and \eqref{prob:Lorentzian}. Our arguments are along the same line as those in section~\ref{sec3}.
We first present a lemma that establishes a relationship between the problems \eqref{prob:LS}, \eqref{prob:Lorentzian} and \eqref{limit_prob}.
\begin{lemma}\label{lem:minseqforP0}
	Consider \eqref{limit_prob} and \eqref{P0} with $ q $ given as in \eqref{prob:LS} or \eqref{prob:Lorentzian}. Then $ \nu_{ncs}^*=\nu^*_d$ if and only if there exists a minimizing sequence of \eqref{P0} that is unbounded.
\end{lemma}
The proof of this lemma is almost identical to that of Lemma~\ref{lem:minseq}. Here we omit the details and only point out a slight difference concerning the derivation of \eqref{Ad=0}. Take \eqref{prob:LS} as an example and let $\{x^t\}$ be an unbounded minimizing sequence of it with $\lim_{t\rightarrow\infty}\frac{x^t}{\|x^t\|} = x^*$ for some $x^*$ satisfying $\|x^*\|=1$. Then one can prove $ Ax^*=0 $ by using the facts that $ \|Ax^t-b\|\leq \sigma $ for all $ t $ and $\|x^t\|\to \infty$. Similar deductions can be done for \eqref{prob:Lorentzian}.

Using Lemma~\ref{lem:minseqforP0}, we can deduce solution existence based on the SSP of $ \ker A $ and the existence of a sparse feasible solution to \eqref{prob:LS} (or \eqref{prob:Lorentzian}). The corresponding arguments are the same as those in Theorem~\ref{thm:solexts_Pold} and we omit the proof for brevity.
\begin{theorem}[Solution existence for \eqref{prob:LS} and \eqref{prob:Lorentzian}]
	Consider \eqref{P0} with $ q $ given as in \eqref{prob:LS} or \eqref{prob:Lorentzian}. Suppose that $\ker A $ has the $ s $-spherical section property and there exists $ \widetilde x\in\R^n $ such that $ \|\widetilde x\|_0< m/s $ and $ q(\widetilde{x}) \leq 0$. Then the optimal value of \eqref{P0} is attainable.
\end{theorem}

\subsection{Optimality conditions}
We discuss first-order necessary optimality conditions for local minimizers. Our analysis is based on the following standard constraint qualifications.
\begin{definition}[{Generalized Mangasarian-Fromovitz constraint qualifications}]
	Consider \eqref{P0}. We say that the general Mangasarian-Fromovitz constraint qualifications (GMFCQ) holds at an $x^*$ satisfying $q(x^*)\le 0$ if the following statement holds:
	\begin{itemize}
		\item If $q(x^*) = 0$, then $0\notin \partial^\circ q(x^*)$.
	\end{itemize}
\end{definition}
The GMFCQ reduces to the standard MFCQ when $ q $ is smooth. One can then see from sections 5.1 and 5.2 of \cite{YuPongLv20} that the GMFCQ holds at every $x$ feasible for \eqref{prob:LS} and \eqref{prob:Lorentzian} for all positive $\sigma$ and $\gamma$, because $A$ is surjective. We next study the GMFCQ for \eqref{prob:rcs}, in which $A$ is also surjective.

\begin{proposition}\label{GMFCQtoitem4}
The GMFCQ holds in the whole feasible set of \eqref{prob:rcs}.
\end{proposition}
\begin{proof}
	It is straightforward to see that the GMFCQ holds for those $ x\in\{x:q(x)<0 \} $ . Then it remains to consider those $x$ satisfying $q(x) = 0$. Let $ q $ be as in \eqref{prob:rcs} and $ \bar x $ satisfy $ q(\bar x)=0 $.
	Notice that a $ \xi\in {\rm  Proj}_S(A\bar x-b) $ takes the following form:
	\[
	\xi_j = \begin{cases}
	[A\bar x-b]_j & {\rm if}\ j\in I^*, \\
	0 & \text{otherwise},
	\end{cases}
	\]
	where $I^*$ is an index set corresponding to the $r$-largest entries (in magnitude).
	Then for any $\xi\in {\rm  Proj}_S(A\bar x-b) $, we have
	\begin{equation}\label{haha}
		\begin{aligned}
			\langle A\bar x-b,\xi\rangle &= \|\xi\|^2,\\
			\|A\bar x-b\|^2 & = \|\xi\|^2 + \|A\bar x - b - \xi\|^2 \\
			&= \|\xi\|^2 + {\rm dist}^2(A\bar x - b,S) \overset{\rm (a)}= \|\xi\|^2 + \sigma^2,
		\end{aligned}
	\end{equation}
	where (a) holds because $0 = q(\bar x) = {\rm dist}^2(A\bar x-b,S) - \sigma^2$.
	Furthermore, since $A$ is surjective, we can deduce from \cite[Example~8.53]{RockWets98}, \cite[Exercise~10.7]{RockWets98} and \cite[Theorem~8.49]{RockWets98} that
	\[
	\partial^\circ q(\bar x) =  {\rm conv} \{2A^T(A\bar x-b-\xi): \xi \in {\rm Proj}_S(A\bar x-b) \}.
	\]
	
	Now, suppose to the contrary that $ 0\in\partial^\circ q(\bar x) $. Using Carath\'{e}odory's theorem, we see that there exist $ \lambda_i\geq 0 $ and $ \xi_i\in{\rm Proj}_S(A\bar x-b) $, $ i=1,\cdots, m+1 $ such that $ \sum_{i=1}^{m+1}\lambda_i = 1 $ and $ \sum_{i=1}^{m+1} \lambda_i A^T(A\bar x -b-\xi_i) = 0 $.
	Since $ A $ is surjective, we then have
	\[
	\sum_{i=1}^{m+1} \lambda_i (A\bar x -b-\xi_i) = 0.
	\]
	Multiplying both sides of the above equality by $(A\bar x - b)^T$, we obtain further that
	\[
	\begin{aligned}
	0& = \sum_{i=1}^{m+1} \lambda_i \langle A\bar x - b, A\bar x -b-\xi_i\rangle = \sum_{i=1}^{m+1} \lambda_i [\|A\bar x-b\|^2 - \langle A\bar x - b,\xi_i\rangle]\\
	& \overset{\rm (a)}= \sum_{i=1}^{m+1} \lambda_i [\|\xi_i\|^2+\sigma^2 - \|\xi_i\|^2]= \sigma^2 > 0,
	\end{aligned}
	\]
	where (a) follows from \eqref{haha} and the fact that $\xi_i\in {\rm Proj}_S(A\bar x-b)$ for each $i$, and the last equality holds because $\sum_{i=1}^{m+1}\lambda_i = 1$. This is a contradiction and thus we must have $ 0\notin\partial^\circ q(\bar x) $. This completes the proof.
\end{proof}

In the next definition, we consider some notions of criticality. The first one is the standard notion of stationarity while the second one involves the Clarke subdifferential.
\begin{definition}\label{def:stationary}
	Consider \eqref{P0}. We say that an $ \bar x\in\R^n $ satisfying $ q(\bar x)\leq 0 $ is
	\begin{enumerate}[{\rm (i)}]
		\item a stationary point of \eqref{P0} if
		\begin{equation}\label{stationary0}
			\textstyle 0\in \partial \left(\frac{\|\cdot\|_1}{\|\cdot\|}+\delta_{[q\leq 0]}(\cdot)\right)(\bar x);
		\end{equation}
		\item a Clarke critical point of \eqref{P0} if there exists $\bar \lambda \ge 0$ such that
		\begin{equation}\label{stationary1}
			\textstyle 0\in \partial \frac{\|\bar x\|_1}{\|\bar x\|} + \bar\lambda \partial^\circ q(\bar x)\ \ {\rm and}\ \ \bar\lambda q(\bar x) =0.
		\end{equation}
	\end{enumerate}
\end{definition}
As mentioned above, Definition~\ref{def:stationary}(i) is standard and it is known that every local minimizer of \eqref{P0} is a stationary point; see \cite[Theorem~10.1]{RockWets98}. We next study some relationships between these notions of criticality, and show in particular that every local minimizer is Clarke critical when the GMFCQ holds.
\begin{proposition}[Stationarity vs Clarke criticality]\label{liftedstationary}
Consider \eqref{P0} and let $\bar x$ be such that $q(\bar x)\le 0$. Then the following statements hold.
\begin{enumerate}[{\rm (i)}]
  \item If $\bar x$ is a stationary point of \eqref{P0} and the GMFCQ holds at $\bar x$, then $\bar x$ is a Clarke critical point.
  \item If $\bar x$ is a Clarke critical point of \eqref{P0} and $q$ is regular at $\bar x$, then $\bar x$ is stationary.
\end{enumerate}
\end{proposition}
\begin{remark}
  Since local minimizers of \eqref{P0} are stationary points, we see from Proposition~\ref{liftedstationary}(i) that when the GMFCQ holds in the whole feasible set, local minimizers are also Clarke critical.
\end{remark}
\begin{proof}
   Suppose that $\bar x$ is a stationary point of \eqref{P0} at which the GMFCQ holds. Then \eqref{stationary0} holds and we consider two cases.
	
	\noindent {\bf Case 1}: $ q(\bar x) <0$. Since $q$ is continuous, \eqref{stationary0} implies $0\in \partial \frac{\|\bar x\|_1}{\|\bar x\|}$ and hence \eqref{stationary1} holds with $\bar\lambda = 0$. Thus, $ \bar x $ is a Clarke critical point. 	
	
	\noindent {\bf Case 2}: $q(\bar x) = 0$. Since the GMFCQ holds for \eqref{P0} at $ \bar x $, we see that $0\notin \partial^\circ q(\bar x)$.
	Then we can deduce from \eqref{stationary0} and \cite[Exercise~10.10]{RockWets98} that
\[
0 \in \partial \frac{\|\bar x\|_1}{\|\bar x\|} + N_{[q\le 0]}(\bar x) \overset{\rm (a)}\subseteq  \partial \frac{\|\bar x\|_1}{\|\bar x\|} + \bigcup_{\lambda\ge 0}\lambda\partial^\circ q(\bar x),
\]
where (a) follows from \cite[Theorem~5.2.22]{BorZhu04}, the first corollary to \cite[Theorem~2.4.7]{Clarke90} and the fact that $0\notin \partial^\circ q(\bar x)$. Thus, \eqref{stationary1} holds with some $\bar\lambda \ge 0$ (recall that $q(\bar x) = 0$), showing that  $ \bar x $ is a Clarke critical point. This proves item (i).

We now prove item (ii). Suppose that $\bar x$ is a Clarke critical point and that $q$ is regular at $\bar x$. Then there exists $\bar \lambda\ge 0$ so that \eqref{stationary1} holds.
We again consider two cases.
	
	\noindent {\bf Case 1:} $\bar \lambda = 0 $. In this case, we see from \eqref{stationary1} that
	$0 \in \partial \frac{\|\bar x\|_1}{\|\bar x\|}$, which implies
\[
\begin{aligned}
0 &\in \partial \frac{\|\bar x\|_1}{\|\bar x\|}\overset{\rm (a)}= \widehat\partial \frac{\|\bar x\|_1}{\|\bar x\|} \subseteq \widehat\partial \frac{\|\bar x\|_1}{\|\bar x\|} + \widehat N_{[q\le 0]}(\bar x) \\
&\overset{\rm (b)}\subseteq \widehat\partial\left(\frac{\|\cdot\|_1}{\|\cdot\|} + \delta_{[q\le 0]}(\cdot)\right)(\bar x)\overset{\rm (c)}\subseteq \partial\left(\frac{\|\cdot\|_1}{\|\cdot\|} + \delta_{[q\le 0]}(\cdot)\right)(\bar x),
\end{aligned}
\]
where (a) follows from \eqref{subdatbarx} and \cite[Corollary~8.11]{RockWets98}, (b) holds thanks to \cite[Corollary~10.9]{RockWets98}, and (c) follows from \cite[Theorem~8.6]{RockWets98}.
Thus, $ \bar x $ is a stationary point.
	
	\noindent {\bf Case 2:} $ \bar \lambda >0 $. In this case, we have from \eqref{stationary1} that $ q(\bar x) = 0 $. Since $ q $ is regular at $\bar x$, we see from \cite[Corollary~8.11]{RockWets98} and the discussion right after \cite[Theorem~8.49]{RockWets98} that
	\begin{equation}\label{eq7}
		\widehat \partial q(\bar x) = \partial q(\bar x) = \partial^\circ q(\bar x).
	\end{equation}
	Now, in view of \eqref{eq7}, $q(\bar x) = 0$ and \cite[Proposition~10.3]{RockWets98}, we have
	\begin{equation}\label{hahahaha}
		\widehat N_{[q\le 0]}(\bar x) \supseteq \bigcup_{\lambda\geq 0}\lambda \widehat \partial q(\bar x) = \bigcup_{\lambda\geq 0}\lambda \partial^\circ q(\bar x).
	\end{equation}
	We then deduce that
	\begin{equation*}
		\partial\left(\frac{\|\cdot\|_1}{\|\cdot\|}+\delta_{[q\leq 0]}(\cdot)\right) (\bar x) \overset{\rm (a)}\supseteq \widehat\partial \frac{\|\bar x\|_1}{\|\bar x\|}+\widehat N_{[q\leq 0]}(\bar x) \overset{\rm (b)}\supseteq \partial \frac{\|\bar x\|_1}{\|\bar x\|}+\bigcup_{\lambda\geq 0}\lambda \partial^\circ q(\bar x),
	\end{equation*}
	where (a) follows from \cite[Theorem~8.6]{RockWets98} and \cite[Corollary 10.9]{RockWets98}, and (b) follows from \eqref{hahahaha}, \eqref{subdatbarx} and \cite[Corollary~8.11]{RockWets98}.
	This together with the definition of Clarke criticality shows that \eqref{stationary0} holds.
	This completes the proof.
\end{proof}

\section{A moving-balls-approximation based algorithm for solving \eqref{P0}}\label{sec6}

In this section, we propose and analyze an algorithm for solving \eqref{P0}, which is an extension of Algorithm~\ref{alg:lou} by incorporating {\em moving-balls-approximation} (MBA) techniques \cite{AusSheTeb10}. Our algorithm, which we call MBA$_{\ell_1/\ell_2}$, is presented as Algorithm~\ref{alg1} below.
\begin{algorithm}
	\caption{MBA$_{\ell_1/\ell_2}$: Moving-balls-approximation based algorithm for \eqref{P0}}\label{alg1}
	\begin{algorithmic}
		\STATE {\bf Step 0.} Choose $x^0$ with $q(x^0)\leq 0 $, $ \alpha>0 $ and $0 < l_{\min} < l_{\max}$. Set $\omega_0 = \|x^0\|_1/\|x^0\|$ and $t=0$.
		
		{\bf Step 1.} Choose $ l^0_t\in[l_{\min},l_{\max}] $ arbitrarily and set $ l_t= l^0_t $. Choose $ \zeta^t\in\partial P_2(x^t) $.
		\begin{enumerate}[({1}a)]
			\item Solve the subproblem
			\begin{equation}\label{subp}
				\begin{array}{rl}
					\widetilde{x}= \displaystyle\argmin_{x\in \R^{n}} & \displaystyle \|x\|_1 - \frac{\omega_t}{\|x^t\|}\langle x,x^t\rangle + \frac{\alpha}2\|x - x^t\|^2\\
					{\rm s.t.}\quad& \displaystyle q(x^t)+ \langle \nabla P_1(x^t)-\zeta^t,x - x^t\rangle + \frac{l_t}{2}\|x - x^t\|^2 \le 0.
				\end{array}
			\end{equation}
			\item If $ q(\widetilde{x})\leq 0 $, go to {\bf Step 2}. Else, update $ l_t\leftarrow 2l_t $ and go to Step (1a).
		\end{enumerate}
		
		{\bf Step 2.} Set $x^{t+1} = \widetilde x$ and compute $\omega_{t+1} = \|x^{t+1}\|_1/\|x^{t+1}\|$. Set $ \bar l_t := l_t $. Update $t\leftarrow t+1$ and go to {\bf Step 1.}
	\end{algorithmic}
\end{algorithm}
Unlike previous works \cite{BolChenPau18,BolPau16,YuPongLv20} that made use of MBA techniques, our algorithm deals with a \emph{fractional} objective and a possibly \emph{nonsmooth} continuous constraint function. Thus, the convergence results in \cite{BolChenPau18,BolPau16,YuPongLv20} cannot be directly applied to analyze our algorithm. Indeed, as we shall see later in section~\ref{sec62}, we need to introduce a new potential function for our analysis to deal with the possibly nonsmooth $q$ in the constraint.

We will show that Algorithm~\ref{alg1} is well defined later, i.e., for each $t\in \mathbb{N}_+$, the subproblem \eqref{subp} has a unique solution for every $l_t$ and the inner loop in {\bf Step 1} terminates finitely. Here, it is worth noting that \eqref{subp} can be efficiently solved using a root-finding procedure outlined in \cite[Appendix~A]{YuPongLv20} since \eqref{subp} takes the form of
\[
\min_x\; \|x\|_1 + \frac{\alpha}2\|x - c^t\|^2\quad \text{s.t.}\; \|x - s^t\|^2\le R_t
\]
for some $ c^t\in\R^n $, $ s^t\in\R^n $ and $ R_t\ge 0 $.

\subsection{Convergence analysis}
In this subsection, we establish subsequential convergence of MBA$_{\ell_1/\ell_2}$ under suitable assumptions. We start with the following auxiliary lemma that concerns well-definedness and sufficient descent. The proof of the sufficient descent property in item (iii) below is essentially the same as \cite[Lemma~1]{WangYanLou19}. We include it here for completeness.
\begin{lemma}[Well-definedness and sufficient descent]\label{welldefine}
	Consider \eqref{P0}. Then the following statements hold:
	\begin{enumerate}[{\rm (i)}]
		\item MBA$_{\ell_1/\ell_2}$ is well defined, i.e., for each $t\in \mathbb{N}_+$, the subproblem \eqref{subp} has a unique solution for every $l_t$ and the inner loop in {\bf Step 1} terminates finitely.
        \item The sequence $\{\bar l_t\}$ is bounded.
		\item Let $\{(x^t,\omega_t)\}$ be the sequence generated by MBA$_{\ell_1/\ell_2}$. Then there exists $ \delta>0 $ such that $ \|x^t\|\geq \delta $ for every $ t\in\mathbb{N}_+ $, and the sequence $\left\{\omega_t\right\}$ satisfies
		\begin{equation}\label{suff_decrs}
			\omega_t - \omega_{t+1}\geq \frac{\alpha}{2\|x^{t+1}\|}\|x^t-x^{t+1}\|^2, \quad   t\in\mathbb{N}_+.
		\end{equation}
	\end{enumerate}
\end{lemma}
\begin{proof}
	Suppose that an $x^t$ satisfying $q(x^t)\le 0$ is given for some $t\in \mathbb{N}_+$. Then $x^t\neq 0$ since $q(0) > 0$. Moreover, for any $l_t> 0$, $ x^t $ is feasible for \eqref{subp} and the feasible set is thus nonempty. Since \eqref{subp} minimizes a strongly convex continuous function over a nonempty closed convex set, it has a unique optimal solution, i.e., $\widetilde x$ exists.
	
	Let $L_p$ be the Lipschitz continuity modulus of $\nabla P_1$. Then we have
	\begin{equation}\label{foritem3}
		\begin{aligned}
			&\textstyle q(\widetilde x) =P_1(\widetilde{x}) - P_2(\widetilde x)\le P_1(x^t)  + \left\langle \nabla P_1(x^t), \widetilde x-x^t\right\rangle + \frac{L_p}2\|\widetilde x-x^t\|^2 -P_2(\widetilde x)\\
            &\textstyle \!\overset{\rm (a)}\leq\! P_1(x^t) - P_2(x^t) +\left\langle \nabla P_1(x^t)-\zeta^t, \widetilde x-x^t\right\rangle + \frac{L_p}2\|\widetilde x-x^t\|^2 \!\overset{\rm (b)}\le\! \frac{L_p - l_t}{2}\|\widetilde x - x^t\|^2,
		\end{aligned}
	\end{equation}
	where (a) holds because of the convexity of $ P_2 $ and the definition of $ \zeta^t $, and (b) follows from the feasibility of $\widetilde x$ for \eqref{subp}. Let $k_{0}\in \mathbb{N}_+$ be such that $L_p - 2^{k_{0}}l_{\min}\le 0$. Then by \eqref{foritem3} and the definition of $l_t$ we see that $q(\widetilde x)\le 0$ after at most $k_{0}$ calls of Step (1b). Moreover, it holds that $\bar l_t \le 2^{k_0}l_{\max}$.
	Therefore, if $q(x^t)\le 0$, then the inner loop of Step 1 stops after at most $k_{0}$ iterations and outputs an $x^{t+1 } $ satisfying $q(x^{t+1})\le 0$ (in particular, $x^{t+1}\neq 0$) with $\bar l_t \le 2^{k_0}l_{\max}$. Since we initialize our algorithm at an $x^0$ satisfying $q(x^0)\le 0 $, the conclusions in items (i) and (ii) now follow from an induction argument. 
	
	Next, we prove item (iii). Since $ q(0)>0 $, we see immediately from the continuity of $ q $ that there exists some $ \delta>0 $ such that $ \|x\|\geq \delta $ whenever $q(x)\le 0$. Thus, $\|x^t\|\ge \delta$ for all $t\in \mathbb{N}_+$, thanks to $q(x^t)\le 0$. Now consider \eqref{subp} with $ l_t = \bar l_t $. Then $ x^t $ is feasible and $ x^{t+1} $ is optimal. This together with the definition of $ \omega_t $ yields
	\[
	\|x^{t+1}\|_1-\frac{\|x^t\|_1}{\|x^t\|^2}\left\langle x^{t+1},x^t\right\rangle+\frac\alpha 2\|x^{t+1}-x^t\|^2\leq \|x^{t}\|_1-\frac{\|x^t\|_1}{\|x^t\|^2}\left\langle x^{t},x^t\right\rangle+\frac\alpha 2\|x^{t}-x^t\|^2 = 0.
	\]
	Dividing both sides of the above inequality by $ \|x^{t+1}\| $ and rearranging terms, we have
	\[
	\frac{\|x^{t+1}\|_1}{\|x^{t+1}\|} + \frac{\alpha}{2\|x^{t+1}\|}\|x^t-x^{t+1}\|^2\leq \frac{\|x^t\|_1}{\|x^t\|^2} \frac{\left\langle x^{t+1}, x^t\right\rangle}{\|x^{t+1}\|}\leq \frac{\|x^t\|_1}{\|x^t\|^2} \frac{\| x^{t+1}\|\|x^t\|}{\|x^{t+1}\|} =  \frac{\|x^t\|_1}{\|x^t\|}.
	\]
	This proves (iii) and completes the proof.
\end{proof}

We next introduce the following assumption.
\begin{assumption}\label{asmp:gmfcq}
	The GMFCQ for \eqref{P0} holds at every point in $[q\le 0]$.
\end{assumption}
Recall from Proposition~\ref{GMFCQtoitem4} and the discussions preceding it that Assumption~\ref{asmp:gmfcq} holds for \eqref{prob:LS}, \eqref{prob:Lorentzian} and \eqref{prob:rcs} since $A$ is surjective.
We next derive the Karush-Kuhn-Tucker (KKT) conditions for \eqref{subp} at every iteration $ t $ under Assumption~\ref{asmp:gmfcq}, which will be used in our subsequent analysis.

\begin{lemma}[KKT conditions for \eqref{subp}]\label{slater-kkt}
	Consider \eqref{P0} and suppose that Assumption~\ref{asmp:gmfcq} holds. Let $\{x^t\}$ be the sequence generated by MBA$_{\ell_1/\ell_2}$. Then the following statements hold:
	\begin{enumerate}[{\rm (i)}]
		\item The Slater's condition holds for the constraint of \eqref{subp} at each $t\in \mathbb{N}_+$.
		\item For each $t\in \mathbb{N}_+$, $ \zeta^t\in\partial P_2(x^t) $ and $l_t>0$, the subproblem \eqref{subp} has a Lagrange multiplier \add{$\lambda_t\geq 0$}. Moreover, if $\widetilde x$ is as in \eqref{subp}, then it holds that
		\begin{align}
			&\lambda_t\left( q(x^t) + \left\langle \nabla P_1(x^t)-\zeta^t, \widetilde x-x^t\right\rangle + \frac{l_t}2\|\widetilde{x} - x^t\|^2\right) = 0, \label{complementary} \\
			&0\in \partial\left\|
			\widetilde x\right\|_1 -\frac{\omega_tx^t}{\|x^t\|}+\lambda_t(\nabla P_1(x^t)-\zeta^t)+(\alpha+\lambda_tl_t)(\widetilde{x}-x^t). \label{firstorder}
		\end{align}
	\end{enumerate}
\end{lemma}
\begin{proof}
	Notice that we can rewrite the feasible set of \eqref{subp} as $ B\left(s^t,\sqrt{R_t}\right)$  with $s^t:=x^t - \frac{1}{l_t}(\nabla P_1(x^t)-\zeta^t)$
	and $ R_t:=\frac{1}{l_t^2}\|\nabla P_1(x^t)-\zeta^t\|^2 - \frac{2}{l_t}q(x^t)$, where $R_t \ge 0$ because $q(x^t)\le 0$. Suppose to the contrary that $ R_t = 0$. Then we have $q(x^t) = 0$ and $
	\nabla P_1(x^t)-\zeta^t= 0$. The latter relation together with \eqref{clarkesubdiffq} implies $ 0\in\partial^\circ q(x^t) $, contradicting the GMFCQ assumption at $x^t $. Thus, we must have $R_t>0$ and hence the Slater's condition holds for \eqref{subp} at the $t^{\rm th}$ iteration.
	
	Since the Slater's condition holds for \eqref{subp}, we can apply \cite[Corollary~28.2.1]{Rock70} and \cite[Theorem~28.3]{Rock70} to conclude that there exists a Lagrange multiplier $\lambda_t$ such that  the relation \eqref{complementary} holds at the $t^{\rm th}$ iteration and  $\widetilde x$ minimizes the following function:
	\begin{align*}
		{\frak L}_t(x) := & \|x\|_1 - \frac{\omega_t}{\|x^t\|}\langle x,x^t\rangle +\frac\alpha 2 \|x-x^t\|^2 \\
		&+ \lambda_t\left(q(x^t) + \left\langle \nabla P_1(x^t)-\zeta^t, x-x^t\right\rangle + \frac{l_t}2 \|x-x^t \|^2\right).
	\end{align*}
	This fact together with \cite[Exercise~8.8]{RockWets98} and \cite[Theorem~10.1]{RockWets98} implies that \eqref{firstorder} holds at the $t^{\rm th}$ iteration. This completes the proof.
\end{proof}

Now we are ready to establish the subsequential convergence of Algorithm~\ref{alg1}. In our analysis, we assume that the GMFCQ holds and that the $ \{x^t\} $ generated by MBA$_{\ell_1/\ell_2}$ is bounded. The latter boundedness assumption was also used in \cite{WangYanLou19} for analyzing the convergence of Algorithm~\ref{alg:lou}. We remark that this assumption is not too restrictive. Indeed, for the sequence $\{x^t\}$ generated by MBA$_{\ell_1/\ell_2}$, in view of Lemma~\ref{welldefine}(i), we know that $q(x^t)\le 0$ for all $t$. Thus, if $q$ is level-bounded, then $\{x^t\}$ is bounded. On the other hand, if $q$ is only known to be bounded from below (as in \eqref{prob:LS}, \eqref{prob:Lorentzian} and \eqref{prob:rcs}) but the corresponding \eqref{P0} is known to have an optimal solution, then one may replace $q(x)$ by the level-bounded function $q_M(x) := q(x) + (\|x\| -M)_+^2$ for a sufficiently large $M$. As long as $M >\|x^*\|$ for some optimal solution $x^*$ of \eqref{P0}, replacing $q$ by $q_M$ in \eqref{P0} will not change the optimal value.
\begin{theorem}[Subsequential convergence of MBA$_{\ell_1/\ell_2}$]\label{thm1}
	Consider \eqref{P0} and suppose that Assumption~\ref{asmp:gmfcq} holds.  Let $ \{(x^t,\zeta^t,\bar l_t)\} $ be the sequence generated by MBA$_{\ell_1/\ell_2}$ and $\lambda_t$ be a Lagrange multiplier of \eqref{subp} with $ l_t = \bar l_t $. Suppose in addition that $ \{x^t\} $ is bounded. Then the following statements hold:
	\begin{enumerate}[\rm (i)]
		\item $ \lim_{t\rightarrow\infty} \|x^{t+1}-x^t\| = 0$;
		\item The sequences $ \{\lambda_t\} $ and $\{\zeta^t\}$ are bounded;
		\item Let $\bar x$ be an accumulation point of $ \{x^t\} $. Then $\bar x$ is a Clarke critical point of \eqref{P0}.
		If $ q $ is also regular at $\bar x$, then $\bar x$ is a stationary point.
	\end{enumerate}
\end{theorem}
\begin{proof}
    Since $ \{x^t\} $ is bounded, there exists $ M>0 $ such that $ \|x^t\|\leq M $ for all $ t\in\mathbb{N}_+ $. Using \eqref{suff_decrs}, we obtain
	\[
	\sum_{t=0}^{\infty} \frac{\alpha}{2M}\|x^t-x^{t+1}\|^2 \leq \omega_0-\liminf_{t\rightarrow\infty} \omega_t \leq  \omega_0,
	\]
	which proves item (i).
	
	Now we turn to item (ii). The boundedness of $ \{\zeta^t\} $ follows from the boundedness of $ \{x^t\} $ and \cite[Theorem~2.6]{Tuy98}. We next prove the boundedness of $\{\lambda_t\}$. Suppose to the contrary that $\{\lambda_{t}\}$ is unbounded. Then there exists a subsequence $ \{\lambda_{t_k}\} $ such that $ \lim_{k\rightarrow\infty} \lambda_{t_k}=\infty$.
	Passing to a subsequence if necessary, we can find subsequences $ \{x^{t_k}\} $ and $ \{\lambda_{t_k}\} $ such that $ \lim_{k\rightarrow\infty} x^{t_k} = x^*$ and $ \lambda_{t_k}> 0 $ for all $ k\in \mathbb{N}_+ $, where the existence of $ x^* $ is due to the boundedness of $ \{x^t\} $. According to \eqref{complementary} and the definition of $ x^{t_k+1} $, we obtain
	\[
	q(x^{t_{k}}) +  \langle  \nabla P_1(x^{t_k})-\zeta^{t_k}, x^{t_k+1} - x^{t_k}\rangle + \frac{\bar l_{t_k}}2\left\|x^{t_k+1} - x^{t_k}\right\|^2 = 0.
	\]
	Since $ \{x^t\} $ is bounded and \add{$ \nabla P_1 $ is Lipschitz continuous}, we then see that $ \{\nabla P_1(x^t)\} $ is bounded. Moreover, $ \{\bar l_{t_k}\} $ is bounded thanks to Lemma~\ref{welldefine}(ii) and we also know that $\{\zeta^t\}$ is bounded. Using these facts, item (i) and the continuity of $ q $, we have upon passing to the limit in the above display that $ q(x^*) =0 $. Since the GMFCQ holds for \eqref{P0} at $ x^* $, we then have $ 0\notin \partial^\circ q(x^*) $.

	Let $ t=t_k $, $ l_t = \bar l_{t_k} $, $  \widetilde{x} =x^{t_k+1} $ in \eqref{firstorder}, and divide both sides of \eqref{firstorder} by $ \lambda_{t_k} $. Then
	\[
	\nabla P_1(x^{t_k})-\zeta^{t_k} \in -\frac 1{\lambda_{t_k}}\partial \|x^{t_k+1}\|_1 + \frac{\omega_{t_k}x^{t_k}}{\lambda_{t_k}\|x^{t_k}\|}-\left(\bar l_{t_k}+\frac{\alpha}{\lambda_{t_k}}\right)(
	x^{t_k+1}-x^{t_k}).
	\]
	Thus, there exists a sequence $ \{\eta^{k}\} $ satisfying $ \eta^k\in\partial \|x^{t_k+1}\|_1 $ and
	\[
	\nabla P_1(x^{t_k})-\zeta^{t_k} = -\frac 1{\lambda_{t_k}}\eta^k + \frac{\omega_{t_k}x^{t_k}}{\lambda_{t_k}\|x^{t_k}\|}-\left(\bar l_{t_k}+\frac{\alpha}{\lambda_{t_k}}\right)(
	x^{t_k+1}-x^{t_k}).
	\]
   Note that $\{\eta^k\}$ is bounded \add{since $ \partial \|x\|_1\subseteq [-1, 1]^n $ for any $ x\in\R^n $.} Moreover, $ \{\omega_{t_k}\} $ is bounded since $ \|x\| \leq \|x\|_1\leq \sqrt{n}\|x\|$ for any $ x\in\R^n $. Furthermore, we have the boundedness of $\{\bar l_{t_k}\}$ from Lemma~\ref{welldefine}(ii). Also recall that $ \lim_{k\rightarrow\infty} \lambda_{t_k}=\infty$ and $\zeta^t\in \partial P_2(x^t)$. Using these together with item (i), we have upon passing to the limit in the above display and invoking the closedness of $\partial P_2$ (see Exercise 8 of \cite[Section~4.2]{BorLew06}) that $\nabla P_1(x^*)\in \partial P_2(x^*)$. This together with \eqref{clarkesubdiffq} further implies $0\in \partial^\circ q(x^*)$, leading to a contradiction. Thus, the sequence $ \{\lambda_t\} $ is bounded.
	
	We now turn to item (iii). Suppose $ \bar x $ is an accumulation point of $ \{x^t\} $ with $ \lim_{j\rightarrow\infty} x^{t_j} = \bar x $ for some convergent subsequence $\{x^{t_j}\}$. Since $ \{ \lambda_t, \bar l_t\} $ and $ \{\zeta^t\} $ are bounded (thanks to Lemma~\ref{welldefine}(ii) and item (ii)), passing to a further subsequence if necessary, we may assume without loss of generality that
	\begin{equation}\label{hahaha}
		\lim_{j\rightarrow\infty}(\lambda_{t_j}, \bar l_{t_j}) = (\bar \lambda, \bar l) \mbox{ for some } \bar\lambda,\ \bar l\geq0,\ \ \ \lim_{j\rightarrow\infty} \zeta^{t_j} = \bar \zeta \mbox{ for some } \bar \zeta\in \partial P_2(\bar x);
	\end{equation}
	here, $ \bar \zeta\in\partial P_2(\bar x) $ because of the closedness of $\partial P_2$ (see Exercise 8 of \cite[Section~4.2]{BorLew06}). On the other hand, according to Lemma~\ref{welldefine}(iii), we have $ \|x^t\|\geq \delta>0 $ for all $t\in \mathbb{N}_+$. This together with the definition of $ \bar x $ yields $ \|\bar x\|\neq 0 $. It then follows that  $ \frac{\|\cdot\|_1}{\|\cdot\|} $ is continuous at $\bar x$. Thus, we have, upon using this fact, the definition of $ \omega_t $, the continuity of $ \nabla P_1 $, the closedness of $ \partial \| \cdot \|_1 $, item (i), \eqref{hahaha}, and passing to the limit as $ j\rightarrow\infty $ in \eqref{firstorder} with $ (\widetilde x,\lambda_t, l_t)=(x^{t_j+1},\lambda_{t_j},  \bar l_{t_j}) $ and $ t = t_j $ that
	\[
	0\in\partial\| \bar x \|_1- \frac{\|\bar x\|_1}{\|\bar x\|^2}\bar x+\bar\lambda
	(\nabla P_1(\bar x)-\bar \zeta).
	\]
	We then divide both sides of the above inclusion by $ \|\bar x\| $ and obtain
	\begin{equation}\label{eq4}
		0\in \frac{1}{\|\bar x\|}\partial\| \bar x \|_1- \frac{\|\bar x\|_1}{\|\bar x\|^3}\bar x+\frac{\bar\lambda}{\|\bar x\|}
		(\nabla P_1(\bar x)-\bar \zeta)= \partial \frac{\|\bar x\|_1}{\|\bar x\|}+\frac{\bar\lambda}{\|\bar x\|}
		(\nabla P_1(\bar x)-\bar \zeta),
	\end{equation}
    where the equality holds due to \eqref{subdatbarx}. In addition, using \eqref{complementary} with $ (\widetilde x,\lambda_t, l_t)=(x^{t_j+1},\lambda_{t_j},  \bar l_{t_j}) $ and $ t = t_j $, we have
	\[
	\lim_{j\rightarrow\infty}\lambda_{t_j}\left[q(x^{t_j})+ \langle \nabla P_1(x^{t_j})-\zeta^{t_j}, x^{t_j+1}-x^{t_j}\rangle+\frac{\bar l_{t_j}}{2}\|x^{t_j+1}-x^{t_j}\|^2\right] = 0.
	\]
	This together with item (i) and \eqref{hahaha} shows that $\bar \lambda q(\bar x) = 0$. Combining this with \eqref{eq4}, $\bar\zeta\in \partial P_2(\bar x)$ (see \eqref{hahaha}), \eqref{clarkesubdiffq} and the fact that $q(\bar x)\le 0$ (because $q(x^t)\le 0$ for all $t$) shows that $\bar x$ is a Clarke critical point.
    Finally, the claim concerning stationarity follows immediately from Proposition~\ref{liftedstationary}. This completes the proof.
\end{proof}

\subsection{Global convergence under KL assumption}\label{sec62}
We now discuss global convergence of the sequence $\{x^t\}$ generated by Algorithm~\ref{alg1}. Our analysis follows the line of analysis in \cite{Attouch09,Attouch10,Attouch13,BolPau16,BolChenPau18,YuPongLv20} and is based on the following auxiliary function:
\begin{equation}\label{barF}
		\widetilde F(x,y,\zeta, w) := \frac{\|x\|_1 }{\|x\|}+\delta_{[\widetilde q \le 0]}(x,y,\zeta,w) + \delta_{\|\cdot\|\ge \rho}(x),
	\end{equation}
	with
	\begin{equation}\label{tildeq}
	\widetilde q(x,y,\zeta,w):= P_1(y) +  \langle\nabla P_1(y), x-y \rangle + P^*_2(\zeta)-  \langle \zeta, x\rangle + \frac w 2 \|x-y\|^2,
	\end{equation}
where $P_1$ and $P_2$ are as in \eqref{P0}, and $\rho > 0$ is chosen such that $\{x:\; q(x)\le 0\}\subset \{x:\; \|x\| > \rho\}$.
Some comments on $\widetilde F$ are in place. First, recall that in the potential function used in \cite{BolPau16} for analyzing their MBA variant, the authors replaced $P_1(x)$ by a quadratic majorant $P_1(y) +  \langle\nabla P_1(y), x-y \rangle + \frac{L_p} 2 \|x-y\|^2$, where $L_p$ is the Lipschitz continuity modulus of $\nabla P_1$. Here, as in \cite{YuPongLv20}, we further introduce the variable $w$ to handle the varying $\bar l_t$. Finally, to deal with the possibly nonsmooth $-P_2$, we replaced $-P_2(x)$ by its majorant $P^*_2(\zeta)-  \langle \zeta, x\rangle$ as in \cite{LiuPong19}.

The next proposition concerns the subdifferential of $\widetilde F$ and will be used for deriving global convergence of the sequence generated by MBA$_{\ell_1/\ell_2}$.
\begin{proposition}\label{lem:subd_barF}
	Consider \eqref{P0} and assume that $P_1$ is twice continuously differentiable. Suppose that Assumption~\ref{asmp:gmfcq} holds. Let $\{(x^t,\zeta^t, \bar{l}_t)\}$ be the sequence generated by MBA$_{\ell_1/\ell_2}$ and suppose that $ \{x^t\} $ is bounded. Let $\widetilde F$ and $\widetilde q$ be given in \eqref{barF} and \eqref{tildeq} respectively.
	Then the following statements hold:
	\begin{enumerate}[\rm (i)]
		\item For any $ t\in\mathbb{N}_+ $, we have $ \widetilde{q}(x^{t+1},x^t, \zeta^t, \bar l_t)\leq 0 $.
		\item There exist $\kappa>0$ and $\underline{t}\in\mathbb{N}_+$ such that
		\begin{equation*}
		{\rm dist}(0, \partial\widetilde F(x^{t+1},x^t,\zeta^t,\bar{l}_t)) \leq \kappa \|
		x^{t+1} - x^t \|\ {\rm \ for \ all \ }t>\underline{t}.
		\end{equation*}
	\end{enumerate}
\end{proposition}
\begin{proof}
	We first observe that
	\begin{align}\label{barht}
	\begin{split}
		&\widetilde q(x^{t+1}, x^t, \zeta^t, \bar l_t) \\
		&= P_1(x^t) + \langle\nabla P_1(x^t), x^{t+1}-x^t\rangle+ P^*_2(\zeta^t)-\langle \zeta^t, x^{t+1}\rangle +\frac{\bar l_t}2 \|x^{t+1}-x^t\|^2\\
	& = P_1(x^t) +\langle \nabla P_1(x^t)-\zeta^t, x^{t+1}-x^t\rangle + P^*_2(\zeta^t)-\langle \zeta^t, x^{t}\rangle +\frac{\bar l_t}2 \|x^{t+1}-x^t\|^2\\
	& \overset{\rm (a)}= P_1(x^t) -P_2(x^t) + \langle \nabla P_1(x^t)-\zeta^t, x^{t+1}-x^t\rangle+\frac{\bar l_t}2 \|x^{t+1}-x^t\|^2 \\
	& = q(x^t) + \langle \nabla P_1(x^t)-\zeta^t, x^{t+1}-x^t\rangle+\frac{\bar l_t}2 \|x^{t+1}-x^t\|^2
	\leq 0,
	\end{split}
	\end{align}
	where (a) follows from \eqref{Young} because $ \zeta^t\in\partial P_2(x^t) $, and the last inequality holds because $ x^{t+1} $ is feasible for \eqref{subp} with $ l_t = \bar l_t $. This proves item (i).
	
	Now, note that $N_{\R_{-}}(\widetilde q(x^{t+1}, x^t,\zeta^t,\bar{l}_t)) = \{0\}$ if $\widetilde q(x^{t+1}, x^t,\zeta^t,\bar{l}_t) < 0$.
	Using this together with \cite[Proposition~10.3]{RockWets98}, we conclude that at any $ (x^{t+1}, x^t, \zeta^t, \bar l_t) $ (regardless of whether $\widetilde q(x^{t+1}, x^t, \zeta^t, \bar l_t) < 0 $ or $\widetilde q(x^{t+1}, x^t, \zeta^t, \bar l_t) = 0 $), the relation
	\[
	\widehat N_{[\widetilde q\leq 0]}(x^{t+1},x^t,\zeta^t,\bar l_t)
	\supseteq \lambda \widehat\partial \widetilde q(x^{t+1},x^t,\zeta^t,\bar l_t)
	\]
	holds for any $ \lambda \in N_{\R_{-}}(\widetilde q(x^{t+1}, x^t,\zeta^t,\bar{l}_t)) $. Thus, for any $ \lambda\in N_{\R_{-}}(\widetilde q(x^{t+1}, x^t,\zeta^t,\bar{l}_t)) $, we have that
	\begin{equation}\label{normal_barq}
		\begin{split}
			&\widehat N_{[\widetilde q\leq 0]}(x^{t+1},x^t,\zeta^t,\bar l_t) \supseteq \lambda \widehat\partial \widetilde q(x^{t+1},x^t,\zeta^t,\bar l_t) \\
			&\overset{\rm (a)} = \begin{bmatrix}
				\lambda[\nabla P_1(x^t) -\zeta^t + \bar l_t(x^{t+1}-x^t)] \\
				\lambda[\nabla^2 P_1(x^t)(x^{t+1}-x^t) - \bar l_t(x^{t+1}-x^t)]\\
				\lambda\partial P^*_2(\zeta^t) - \lambda x^{t+1}\\
				\frac \lambda 2\left\|x^{t+1}-x^t\right\|^2
			\end{bmatrix}  \overset{\rm (b)}\ni
			\begin{bmatrix}
				\lambda V_1^t \\
				\lambda V^t_2\\
				\lambda(x^t-x^{t+1})\\
				\frac \lambda 2\left\|x^{t+1}-x^t\right\|^2
			\end{bmatrix}
		\end{split},
	\end{equation}
	with
	\begin{equation}\label{v1v2}
		\begin{array}{l}
			V_1^t := \nabla P_1(x^t) -\zeta^t + \bar l_t(x^{t+1}-x^t),\\ [3 pt]
			V_2^t := \nabla^2 P_1(x^t)(x^{t+1}-x^t) - \bar l_t(x^{t+1}-x^t),
		\end{array}
	\end{equation}
	where (a) uses the definition of $ \widetilde q $, \cite[Exercise~8.8(c)]{RockWets98}, \cite[Proposition~10.5]{RockWets98} and \cite[Proposition~8.12]{RockWets98} (so that $\partial P^*_2(\zeta^t) = \widehat \partial P^*_2(\zeta^t)$), and (b) uses \eqref{Young} and the fact that $ \zeta^t\in\partial P_2(x^t) $. On the other hand, we have from \cite[Theorem~8.6]{RockWets98} that
	\begin{equation}\label{partialbarF}
		\begin{aligned}
			&\partial \widetilde F(x^{t+1},x^t,\zeta^t,\bar l_t)  \supseteq \widehat\partial \widetilde F(x^{t+1},x^t,\zeta^t,\bar l_t)\\
			& \overset{\rm (a)} \supseteq \begin{bmatrix}
				\frac{1}{\|x^{t+1}\|}\partial \left\|x^{t+1}\right\|_1-\frac{\|x^{t+1}\|_1}{\|x^{t+1}\|^3}x^{t+1}\\
				0\\
				0\\
				0
			\end{bmatrix} + \widehat{N}_{[\widetilde q\leq 0]}(x^{t+1},x^t,\zeta^t,\bar l_t),
		\end{aligned}
	\end{equation}
	where (a) uses \cite[Corollary 10.9]{RockWets98}, \eqref{subdatbarx} and \cite[Corollary~8.11]{RockWets98}, and the facts that $ \widehat\partial \delta_{[\widetilde q\leq 0]} (x^{t+1},x^t,\zeta^t,\bar l_t) = \widehat{N}_{[\widetilde q\leq 0]}(x^{t+1},x^t,\zeta^t,\bar l_t)$ and $\widehat N_{\|\cdot\|\ge \rho}(x^{t+1}) = \{0\}$.
	
	Let $ \lambda_t\geq 0 $ be a Lagrange multiplier of \eqref{subp} with $ l_t = \bar l_t $, which exists thanks to Lemma~\ref{slater-kkt}. In view of the inequality and the last equality in \eqref{barht} and using \eqref{complementary} with $ (\widetilde{x}, l_t) = (x^{t+1}, \bar l_t) $, we deduce that $ \lambda_t\in N_{\R_-}(\widetilde q(x^{t+1}, x^t,\zeta^t,\bar{l}_t))$, which in turn implies that $ \frac{\lambda_t}{\|x^{t+1}\|}\in N_{\R_-}(\widetilde q(x^{t+1}, x^t,\zeta^t,\bar{l}_t)) $. We can hence let $ \lambda = \frac{\lambda_t}{\|x^{t+1}\|} $ in \eqref{normal_barq} to obtain an element in $\widehat N_{[\widetilde q\leq 0]}(x^{t+1},x^t,\zeta^t,\bar l_t)$. Plugging this particular element into \eqref{partialbarF} yields
	\begin{equation}\label{bounded_subdbarF0}
			\partial \widetilde F(x^{t+1},x^t,\zeta^t,\bar l_t)
			\supseteq \begin{bmatrix}
				\frac{1}{\|x^{t+1}\|} \partial \left\|x^{t+1}\right\|_1-\frac{\|x^{t+1}\|_1}{\|x^{t+1}\|^3}x^{t+1}+\frac{\lambda_t}{\|x^{t+1}\|} V_1^t\\
				\frac{\lambda_t}{\|x^{t+1}\|} V_2^t\\
				\frac{\lambda_t}{\|x^{t+1}\|}(x^t-x^{t+1})\\
				\frac{\lambda_t}{2\|x^{t+1}\|}\left\|x^{t+1}-x^t\right\|^2
			\end{bmatrix},
	\end{equation}
	where $V_1^t$ and $ V_2^t $ are given in \eqref{v1v2}. On the other hand, applying \eqref{firstorder} with $ (\widetilde{x}, l_t) = (x^{t+1}, \bar l_t) $ and recalling that $\omega_t = \|x^t\|_1/\|x^t\|$, we obtain
    \begin{equation}\label{hehehaha}
    \begin{aligned}
	\partial \|x^{t+1}\|_1&\ni \frac{\|x^t\|_1}{\|x^t\|^2}x^t -\lambda_{t}(\nabla P_1(x^t)-\zeta^t)-(\alpha+\lambda_t\bar l_t)(x^{t+1}-x^t) \\
    &=  \frac{\|x^t\|_1}{\|x^t\|^2}x^t -\lambda_{t}V^t_1-\alpha (x^{t+1}-x^t).
    \end{aligned}
	\end{equation}
    Combining \eqref{bounded_subdbarF0} and \eqref{hehehaha}, we see further that
    \begin{equation}\label{bounded_subdbarF}
			\partial \widetilde F(x^{t+1},x^t,\zeta^t,\bar l_t)
			\ni \begin{bmatrix}
				J_1^t\\
				\frac{\lambda_t}{\|x^{t+1}\|}V_2^t\\
				\frac{\lambda_t}{\|x^{t+1}\|}(x^t-x^{t+1})\\
				\frac{\lambda_t}{2\|x^{t+1}\|}\left\|x^{t+1}-x^t\right\|^2
			\end{bmatrix},
	\end{equation}
where
	\begin{equation*}
		J_1^t := \frac{1}{\|x^{t+1}\|}  \left( \frac{\|x^t\|_1}{\|x^t\|^2}x^t - \frac{\|x^{t+1}\|_1}{\|x^{t+1}\|^2}x^{t+1} \right) - \frac{\alpha}{\|x^{t+1}\|} (x^{t+1}-x^t).
	\end{equation*}

    Next, recall from Lemma~\ref{welldefine}(iii) that
	\begin{equation}\label{eq9}
		\|x^{t+1}\|\ge \delta, \mbox{ for all } t\in\mathbb{N}_+.
	\end{equation}
	Using this together with our assumption that $ \{x^t\} $ is bounded, we see that there exists $ L_1>0 $ such that
	\begin{equation*}
		\left\| \frac{\|x^t\|_1}{\|x^t\|^2}x^t - \frac{\|x^{t+1}\|_1}{\|x^{t+1}\|^2}x^{t+1} \right\| \leq L_1 \|x^{t+1}-x^t\|\ \ \mbox{for all }t.
	\end{equation*}
	Combining the above three displays, we deduce that
	\begin{equation}\label{eq0}
		\left\|J_1^t\right\|\leq \frac{L_1+\alpha}{\delta}\|x^{t+1}-x^t\|.
	\end{equation}
	On the other hand, one can see from \eqref{eq9}, the definition of $ V_2^t $ (see \eqref{v1v2}), the boundedness of $ \{\lambda_t, \bar l_t\} $ (see Theorem~\ref{thm1}(ii) and Lemma~\ref{welldefine}(ii)), the continuity of $ \nabla^2 P_1 $ and the boundedness of $ \{x^t\} $ that there exist $ L_2>0 $ and $ L_3>0 $ such that
	\begin{equation}\label{eq5}
		\frac{\lambda_t}{\|x^{t+1}\|}\leq \frac{L_2}{\delta} \text{ and } \left\| \frac{\lambda_t}{\|x^{t+1}\|}V_2^t \right\|\leq  L_3 \|x^{t+1}-x^t\|.
	\end{equation}
	Moreover, we can see from Theorem~\ref{thm1}(i) that there exists $ \underline{t}\in \mathbb{N}_+ $ such that
	\[
	\|x^{t+1} - x^t\|^2 \leq \|x^{t+1} - x^t\|
	\]
	whenever $ t\geq \underline{t} $. Now we can conclude from \eqref{bounded_subdbarF}, \eqref{eq0}, \eqref{eq5} and the above display that there exists $ \kappa>0 $ such that
	\[
	{\rm dist}(0, \partial\widetilde F(x^{t+1},x^t,\zeta^t,\bar{l}_t)) \leq \kappa \|
	x^{t+1} - x^t \|
	\]
	for all $t\ge \underline{t}$. This completes the proof.
\end{proof}

When the sequence $\{x^t\}$ generated by MBA$_{\ell_1/\ell_2}$ is bounded, one can show that the set of accumulation points $\Omega$ of $\{(x^{t+1},x^t,\zeta^t,\bar l_t)\}$ is compact. This together with Lemma~\ref{welldefine}(iii) and the continuity of $\widetilde F$ on its domain shows that $\widetilde F$ is constant on $\Omega\subseteq {\rm dom}\,\partial \widetilde F$. Using this together with
Proposition~\ref{lem:subd_barF} and Lemma~\ref{welldefine}(iii), one can prove the following convergence result by imposing additional KL assumptions on $\widetilde F$. The proof is standard and follows the line of arguments as in \cite{Attouch09,Attouch10,Attouch13,BolSabTeb14,LiuPong19,WenChenPong18}. We omit the proof here for brevity.

\begin{theorem}[Global convergence and convergence rate of MBA$_{\ell_1/\ell_2}$]\label{thm:useful}
Consider \eqref{P0} and assume that $P_1$ is twice continuously differentiable. Suppose that Assumption~\ref{asmp:gmfcq} holds. Let $ \{x^t\} $ be the sequence generated by MBA$_{\ell_1/\ell_2}$ and assume that $  \{x^t\}$ is bounded. If $\widetilde F$ in \eqref{barF} is a KL function, then $ \{x^t\} $ converges to a Clarke critical point $ x^* $ of \eqref{P0} ($ x^* $ is stationary if $ q $ is in addition regular at $x^*$). Moreover, if $ \widetilde F $ is a KL function with exponent $ \theta\in[0,1) $, then the following statements hold:
	\begin{enumerate}[\rm (i)]
		\item If $ \theta = 0 $, then $ \{x^t\} $ converges finitely.
		\item If $ \theta\in (0, \frac12] $, then there exist $ c_0>0 $, $ Q_1\in(0,1) $ and $ \underline t\in\mathbb{N}_+ $ such that
		\[
		\|x^t-x^*\|\leq c_0Q_1^t \,\mbox{ for }\, t>\underline t.
		\]
		\item If $ \theta\in(\frac12, 1) $, then there exist $ c_0>0 $ and $ \underline{t}\in\mathbb{N}_+ $ such that
		\[
		\|x^t-x^*\|\leq c_0t^{-\frac{1-\theta}{2\theta-1}} \,\mbox{ for }\, t>\underline t.
		\]
	\end{enumerate}
\end{theorem}
\begin{remark}[KL property of $\widetilde F$ corresponding to \eqref{prob:LS}, \eqref{prob:Lorentzian} and \eqref{prob:rcs}]\label{rem:useful}
	\begin{enumerate}[{\rm (i)}]
		\item In both \eqref{prob:LS} and \eqref{prob:Lorentzian}, we have $q = P_1 $ being analytic and $P_2^* = \delta_{\{0\}}$. Hence $ \widetilde F $ becomes $ \widetilde F(x,y,\zeta,w) = \frac{\|x\|_1}{\|x\|} + \delta_{\Delta}(x,y,\zeta,w) $ with $\Delta = \{(x,y,\zeta,w):\; P_1(y)+\langle \nabla P_1(y), x- y\rangle +\frac{w}2\|x- y\|^2 \le 0, \zeta = 0, \|x\|\ge \rho\}$. Hence, the graph of $\widetilde F$ is
\[
\left\{(x,y,\zeta,w,z):\;
\begin{array}{l}
  \|x\|_1 = z\|x\|,\ \ \ \ \|x\|\ge \rho,\ \ \ \ \zeta = 0,\\
  P_1(y)+\langle \nabla P_1(y), x- y\rangle +\frac{w}2\|x- y\|^2 \le 0.
\end{array}
\right\},
\]
which is semianalytic \cite[Page~596]{FaPang03}. This means that $ \widetilde F $ is subanalytic \cite[Definition~6.6.1]{FaPang03}. Moreover, the domain of $ \widetilde F $ is closed and $\widetilde F|_{{\rm dom}\,\widetilde F}$ is continuous. Therefore, $ \widetilde F $ satisfies the KL property according to \cite[Theorem~3.1]{BolDan07}.
		
		\item For \eqref{prob:rcs}, first note that $ P_2 $ is a convex piecewise linear-quadratic function (see, for example, the proof of \cite[Theorem~5.1]{LiuPong19}). Then $ P_2^* $ is also piecewise linear-quadratic function thanks to  \cite[Theorem~11.14]{RockWets98}. Thus, one can see that  $\widetilde q  $ corresponding to \eqref{prob:rcs} is semialgebraic and so is the set $\Theta = \{(x, y, \zeta, w): \; \widetilde q(x,y, \zeta, w)\leq 0\} $. Therefore $\widetilde F$ is semialgebraic as the sum of the semialgebraic functions $x\mapsto \frac{\|x\|_1}{\|x\|} + \delta_{\|\cdot\|\ge\rho}(x)$ and $\delta_\Theta$, and is hence a KL function \cite{Attouch10}.
	\end{enumerate}
\end{remark}

Using Theorem~\ref{thm:useful}, Remark~\ref{rem:useful}, Proposition~\ref{GMFCQtoitem4} and the discussions preceding it, and recalling that continuously differentiable functions are regular, we have the following immediately corollary.
\begin{corollary}[Global convergence of MBA$_{\ell_1/\ell_2}$ for \eqref{prob:LS}, \eqref{prob:Lorentzian} and \eqref{prob:rcs}]
  \begin{enumerate}[{\rm (i)}]
    \item Suppose that MBA$_{\ell_1/\ell_2}$ is applied to \eqref{prob:LS} or \eqref{prob:Lorentzian}. Then the sequence generated converges to a stationary point of the problem if the sequence is bounded.
    \item Suppose that MBA$_{\ell_1/\ell_2}$ is applied to \eqref{prob:rcs}. Then the sequence generated converges to a Clarke critical point of the problem if the sequence is bounded.
  \end{enumerate}
\end{corollary}

\section{Numerical simulations}\label{sec7}
In this section, we perform numerical experiments on solving random instances of \eqref{prob:LS}, \eqref{prob:Lorentzian} and \eqref{prob:rcs} by MBA$_{\ell_1/\ell_2}$.
All numerical experiments are performed in MATLAB 2019b on a 64-bit PC with an Intel(R) Core(TM) i7-6700 CPU (3.40GHz) and 32GB of RAM.

We set $ l_{\min} = 10^{-8} $, $ l_{\max} = 10^8 $ and $\alpha = 1 $ in MBA$_{\ell_1/\ell_2}$. We let $ l_0^0 = 1 $ and compute, for each $ t\geq 1 $,
\[
l_t^0 = \begin{cases} \max\left\{l_{\min}, \min\left\{\frac{\langle d_x^t,d_g^t\rangle}{\|d_x^t\|^2}, l_{\max}\right\}\right\}  &{\rm if\ } \langle d_x^t ,d_g^t \rangle \ge 10^{-12},\\
\max\left\{l_{\min}, \min\left\{\frac{l_{t-1}}2, l_{\max}\right\}\right\} &{\rm otherwise},
\end{cases}
\]
where $d_x^t = x^t - x^{t-1}$ and $d_g^t=\xi^t - \xi^{t-1} $ with $ \xi^t = \nabla P_1(x^t) - \zeta^t$: specifically, $\zeta^t = 0$ when solving \eqref{prob:LS} and \eqref{prob:Lorentzian}, while for \eqref{prob:rcs}, we pick any $ \zeta^t\in {\rm Proj}_S(Ax^t-b) $, which can be obtained by finding the largest $ r $ entries of $ Ax^t-b $.

We initialize MBA$_{\ell_1/\ell_2}$ at some feasible point $ x_{\rm feas} $ and terminate MBA$_{\ell_1/\ell_2}$ when
\begin{equation}\label{termination}
\|x^{t}-x^{t-1}\|\leq tol\cdot \max\{\|x^t\|, 1\};
\end{equation}
we will specify the choices of $  x_{\rm feas} $ and $tol$ in each of the subsections below.


\subsection{Robust compressed sensing problems \eqref{prob:rcs}}
We generate a sensing matrix $ A\in\R^{(p+\iota)\times n} $ with i.i.d standard Gaussian entries and then normalize each column of $ A $. Next, we generate the original signal $x_{\rm orig}\in\mathbb{R}^{n}$  as a $k$-sparse  vector with $k$ i.i.d standard Gaussian entries at random (uniformly chosen) positions. We then generate a vector $ z_\iota\in\R^\iota $ with i.i.d. standard Gaussian entries, and set $ z\in\R^{p+\iota} $ to be a vector with the first $ p $ entries being zero and the last $ \iota $ entries being $ 2\,{\rm sign}(z_\iota) $. The vector $ b $ in \eqref{prob:rcs} is then generated as $ b = Ax_{\rm orig}-z+0.01\varepsilon $, where $ \varepsilon\in\R^{p+\iota} $ has i.i.d. standard Gaussian entries. Finally, we set $ \sigma=1.2\|0.01\varepsilon\| $ and $r = 2 \iota$. In MBA$_{\ell_1/\ell_2}$, we set $x_{\rm feas} = A^\dagger b $,\footnote{We compute $ A^\dagger b $ via the MATLAB commands {\sf [Q,R] = qr(A',0); xfeas = Q*(R'$\backslash$b)}. } and $ tol = 10^{-6} $ in \eqref{termination}.

In our numerical tests, we consider $ (n,p, k, \iota) = (2560i, 720i, 80i, 10i) $ with $ i\in\{2, 4, 6, 8, 10\} $. For each $ i $, we generate 20 random instances as described above. The computational results are shown in Table~\ref{table:rcs}. We present the time $ t_{\rm qr} $ for the (reduced) QR decomposition when generating $ x_{\rm feas} $, the CPU times $t_{\rm mba}$ and $t_{\rm sum}$,\footnote{$t_{\rm mba}$ is the run time of MBA$_{\ell_1/\ell_2}$, while $t_{\rm sum}$ includes the run time of MBA$_{\ell_1/\ell_2}$, the time for performing (reduced) QR factorization on $A^T$ and the time for computing $Q(R^{-1})^T b$.} the recovery error ${\rm RecErr} = \frac{\|x_{\rm out}-x_{\rm orig}\|}{\max\{1,\|x_{\rm orig}\|\}} $, and the $ {\rm Residual} = {\rm dist}^2(Ax_{\rm out}-b, S)-\sigma^2 $, averaged over the 20 random instances, where $x_{\rm out}$ is the approximate solution returned by MBA$_{\ell_1/\ell_2}$. We see that $x_{\rm orig}$ are approximately recovered in a reasonable period of time.
\begin{table}[h]
	\caption{Random tests on robust compressed sensing}\label{table:rcs}
	\begin{center}
		{\footnotesize
			\begin{tabular}{cccccc}\hline
			   $ i $  &     $t_{\rm qr}$ &  $t_{\rm mba} (t_{\rm sum})$ &  $ {\rm RecErr} $ & Residual \\ \hline
			     2 &    0.5 &    1.2 (   1.7)  & 3.3e-02 & -3e-11  \\
			     4 &    3.1 &    4.1 (   7.2)  & 3.3e-02 & -5e-11  \\
			     6 &    9.8 &    8.3 (  18.1)  & 3.3e-02 & -9e-11  \\
			     8 &   24.0 &   14.3 (  38.4)  & 3.3e-02 & -1e-10  \\
			    10 &   43.6 &   21.5 (  65.3)  & 3.3e-02 & -2e-10  \\ \hline			
			\end{tabular}
		}
	\end{center}
\end{table}

\subsection{CS problems with Cauchy noise \eqref{prob:Lorentzian}}
Similar to the previous subsection, we generate the sensing matrix $ A \in\R^{m\times n} $  with i.i.d standard Gaussian entries and then normalize each column of $ A $. We then generate the original signal $x_{\rm orig}\in\mathbb{R}^{n}$ as a $k$-sparse vector with $k$ i.i.d standard Gaussian entries at random (uniformly chosen) positions. However, we generate $b$ as $ b = Ax_{\rm orig} + 0.01 \varepsilon $ with $ \varepsilon_i\sim {\rm Cauchy}(0,1) $, i.e., $\varepsilon_i = \tan(\pi(\widetilde{\varepsilon}_i - 1/2))$ for some random vector $\widetilde \varepsilon\in\R^m$  with i.i.d. entries uniformly chosen in $[0,1]$. Finally, we set $ \gamma = 0.02 $ and $ \sigma = 1.2\|0.01\varepsilon\|_{LL_2,\gamma}$.

We compare the $ \ell_1 $ minimization model (which minimizes $\ell_1$ norm in place of $\ell_1/\ell_2$ in \eqref{prob:Lorentzian}; see \cite[Eq.~(5.8)]{YuPongLv20} with $\mu = 0$) with our $ \ell_1/\ell_2 $ model. We use ${\rm SCP}_{\rm ls}$ in \cite{YuPongLv20} for solving the $ \ell_1 $ minimization model. We use the same parameter settings for ${\rm SCP}_{\rm ls}$ as in \cite[Section~5]{YuPongLv20}, except that we terminate ${\rm SCP}_{\rm ls}$ when \eqref{termination} is satisfied with $ tol=10^{-6} $ \add{in Table~\ref{table:Lorentzian1}}. We initialize MBA$_{\ell_1/\ell_2}$ at the approximate solution $ x_{\rm scp} $ given by ${\rm SCP}_{\rm ls}$, and terminate MBA$_{\ell_1/\ell_2}$ when \eqref{termination} is satisfied with $ tol = 10^{-6} $.

In our numerical experiments, we consider $ (n,m, k) = (2560i, 720i, 80i) $ with $ i\in\{2, 4, 6, 8, 10\} $. For each $ i $, we generate 20 random instances as described above. Our computational results are presented in Table~\ref{table:Lorentzian1}, which are averaged over the 20 random instances. Here we show the CPU time $ t_{\rm qr} $ for performing (reduced) QR decomposition on $A^T$, the CPU time,\footnote{For MBA$_{\ell_1/\ell_2}$, the time in parenthesis is the total run time including the time for computing the initial point $A^\dagger b$ for SCP$_{\rm ls}$ and the run times of SCP$_{\rm ls}$ and MBA$_{\ell_1/\ell_2}$, the time without parenthesis is the actual run time of MBA$_{\ell_1/\ell_2}$ starting from $x_{\rm feas} = x_{\rm scp}$.} the recovery error $ {\rm RecErr} = \frac{\|x_{\rm out}-x_{\rm orig}\|}{\max\{1,\|x_{\rm orig}\|\}} $ and the residual $ {\rm Residual} =\|Ax_{\rm out}-b\|_{LL_2,\gamma} - \sigma $ of both ${\rm SCP}_{\rm ls}$ and MBA$_{\ell_1/\ell_2}$, where $x_{\rm out}$ is the approximate solution returned by the respective algorithm. We see that the recovery error is significantly improved by solving the nonconvex model.

\add{Finally, as suggested by one reviewer, we investigate the effect of initialization on the performance of  MBA$_{\ell_1/\ell_2}$. Specifically, we test ${\rm SCP}_{\rm ls}$ and MBA$_{\ell_1/\ell_2}$ on the same set of instances used in Table~\ref{table:Lorentzian1}, but terminate ${\rm SCP}_{\rm ls}$ when \eqref{termination} is satisfied with $ tol = 10^{-3} $. We then initialize MBA$_{\ell_1/\ell_2}$ at the approximate solution returned by ${\rm SCP}_{\rm ls}$, and terminate MBA$_{\ell_1/\ell_2}$ when \eqref{termination} is satisfied with $ tol=10^{-6} $. The computational results are presented in Table~\ref{table:Lorentzian2}. Not too surprisingly, we can see that MBA$_{\ell_1/\ell_2}$ can result in large recovery errors with this initialization, though the recovery errors may still be small sometimes (see $i = 6$). Thus, the performance of MBA$_{\ell_1/\ell_2}$ is quite sensitive to its initialization.}

\begin{table}[h]
	\caption{Random tests on CS problems with Cauchy noise (\add{$ tol = 10^{-6} $ for $ {\rm SCP}_{\rm ls} $)}}\label{table:Lorentzian1}
	\begin{center}
		{\footnotesize
			\begin{tabular}{ccccccccc}\hline
			\multirow{2}{*}{ $ i $ } & \multirow{2}{*}{$ t_{\rm qr} $} & \multicolumn{2}{c}{CPU} & \multicolumn{2}{c}{RecErr}  & \multicolumn{2}{c}{Residual} \\ \cmidrule(lr){3-4} \cmidrule(lr){5-6} \cmidrule(lr){7-8}
		    &   &   $ {\rm SCP}_{\rm ls} $ &  MBA$_{\ell_1/\ell_2}$  &     $ {\rm SCP}_{\rm ls} $ &   MBA$_{\ell_1/\ell_2}$  &   $ {\rm SCP}_{\rm ls} $ &   MBA$_{\ell_1/\ell_2}$   \\ \hline
                 2 &    0.5   &   10.0 &     0.6 (  11.1) & 1.3e-01 & 6.5e-02 & -2e-07& -8e-08  \\
                 4 &    3.0   &   52.4 &     2.0 (  57.5) & 1.3e-01 & 6.6e-02 & -6e-07& -2e-07  \\
                 6 &    9.4   &   87.3 &     4.1 ( 100.9) & 1.3e-01 & 6.6e-02 & -9e-07& -2e-07  \\
                 8 &   23.4   &  281.6 &     7.0 ( 312.1) & 1.3e-01 & 6.5e-02 & -1e-06& -3e-07  \\
                10 &   42.4   &  285.5 &    11.4 ( 339.5) & 1.3e-01 & 6.5e-02 & -2e-06& -4e-07  \\ \hline
			\end{tabular}
		}
	\end{center}
\end{table}

\begin{table}[h]
	\caption{\textcolor{blue}{Random tests on CS problems with Cauchy noise ($ tol = 10^{-3} $ for $ {\rm SCP}_{\rm ls} $)}\label{table:Lorentzian2}}
	\begin{center}
		{\footnotesize \color{blue}
			\begin{tabular}{ccccccccc}\hline
				\multirow{2}{*}{ $ i $ } & \multirow{2}{*}{$ t_{\rm qr} $} & \multicolumn{2}{c}{CPU} & \multicolumn{2}{c}{RecErr}  & \multicolumn{2}{c}{Residual} \\ \cmidrule(lr){3-4} \cmidrule(lr){5-6} \cmidrule(lr){7-8}
				&   &   $ {\rm SCP}_{\rm ls} $ &  MBA$_{\ell_1/\ell_2}$  &     $ {\rm SCP}_{\rm ls} $ &   MBA$_{\ell_1/\ell_2}$  &   $ {\rm SCP}_{\rm ls} $ &   MBA$_{\ell_1/\ell_2}$   \\ \hline
			2  &   0.5  &    3.0 &    50.8 (  54.3) & 1.8e+00 & 1.6e+00 & -3e+01  & -6e-05 \\
			4  &   3.0  &   11.8 &   457.6 ( 472.5) & 4.3e+00 & 4.2e+00 & -1e+02  & -5e-04 \\
			6  &   9.5  &   30.5 &     4.9 (  44.9) & 2.1e-01 & 6.6e-02 & -9e-01  & -2e-07 \\
			8  &  22.9  &   37.7 &    78.5 ( 139.2) & 9.7e+00 & 9.6e+00 & -6e+01  & -9e-03 \\
			10  &  41.5  &   71.9 &  3164.0 (3277.6) & 2.1e+00 & 1.7e+00 & -1e+02  & -2e-04\\ \hline
			\end{tabular}
		}
	\end{center}
\end{table}

\subsection{Badly scaled CS problems with Gaussian noise \eqref{prob:LS}}
In this section, we generate test instances similar to those in \cite{WangYanLou19}. Specifically, we first generate $ A = [a_1,\cdots, a_n]\in\R^{m\times n} $ with
\[
a_j = \frac{1}{\sqrt{m}}\cos\left(\frac{2\pi wj}{F}\right), \quad j=1, \cdots, n,
\]
where $ w\in\R^m $ have i.i.d. entries uniformly chosen in $[0,1]$. Next, we generate the original signal $x_{\rm orig}\in\mathbb{R}^{n}$ using the following MATLAB command:
\vskip 1mm
\begin{verbatim}
           I = randperm(n); J = I(1:k); xorig = zeros(n,1);
	           xorig(J) = sign(randn(k,1)).*10.^(D*rand(k,1));
\end{verbatim}
\vskip 1mm
We then set $ b = Ax_{\rm orig} + 0.01 \varepsilon $, where $\varepsilon\in\R^m $  has i.i.d standard Gaussian entries. Finally, we set $ \sigma=1.2\|0.01\varepsilon\| $.

We compare the $ \ell_1 $ minimization model (which minimizes $\ell_1$ norm in place of $\ell_1/\ell_2$ in \eqref{prob:LS}; see \cite[Eq.~(5.5)]{YuPongLv20} with $\mu = 0$) with our $ \ell_1/\ell_2 $ model. The $ \ell_1 $ minimization model is solved via SPGL1 \cite{VaFr08} (version 2.1) using default settings. The initial point for MBA$_{\ell_1/\ell_2}$ is generated from the approximate solution $x_{\rm spgl1}$ of SPGL1 as follows: Specifically, since $x_{\rm spgl1}$ may violate the constraint slightly, we set the initial point of MBA$_{\ell_1/\ell_2}$ as
\[
x_{\rm feas} =
\begin{cases}
  A^\dagger b + \sigma \frac{x_{\rm spgl1} - A^\dagger b}{\|A x_{\rm spgl1} - b\|} & {\rm if}\ \|A x_{\rm spgl1} - b\| > \sigma,\\
  x_{\rm spgl1}                                                                    & {\rm otherwise}.
\end{cases}
\]
We terminate MBA$_{\ell_1/\ell_2}$ when \eqref{termination} is satisfied with $ tol = 10^{-8} $.

In our numerical tests, we set $ n = 1024 $, $ m = 64 $ and consider  $ k\in\{8, 12\} $, $ F\in\{5, 15\} $ and $ D\in\{2, 3\} $. For each $ (k, F, D) $, we generate 20 random instances as described above. We present the computational results (averaged over the 20 random instances) in Table~\ref{table:BP}. Here we show the CPU time,\footnote{For MBA$_{\ell_1/\ell_2}$, the time in parenthesis is the total run time including the time for computing the feasible point $A^\dagger b$ and the run times of SPGL1 and MBA$_{\ell_1/\ell_2}$, the time without parenthesis is the actual run time of MBA$_{\ell_1/\ell_2}$ starting from $x_{\rm feas}$.} the recovery error $ {\rm RecErr} = \frac{\|x_{\rm out}-x_{\rm orig}\|}{\max\{1,\|x_{\rm orig}\|\}} $, the $ {\rm Residual} =\|Ax_{\rm out}-b\|^2-\sigma^2 $ of both SPGL1 and MBA$_{\ell_1/\ell_2}$, where $x_{\rm out}$ is the approximate solution returned by the respective algorithm. We again observe that the recovery error is significantly improved (on average) by solving the nonconvex model in most instances, except when $(k,F,D)=(12,15,3)$. In this case, we see that the $x_{\rm spgl1}$ can be highly infeasible and thus the starting point $x_{\rm feas}$ provided to MBA$_{\ell_1/\ell_2}$ may not be a good starting point. This might explain the relatively poor performance of MBA$_{\ell_1/\ell_2}$ in this case.

\begin{table}
	\caption{Random tests on badly scaled CS problems with Gaussian noise}\label{table:BP}
	\begin{center}
		{\scriptsize
			\begin{tabular}{ccc ccc ccc ccc}\hline
			 \multirow{2}{*}{$ k $} & \multirow{2}{*}{$ F $ } & \multirow{2}{*}{ $ D $} & \multicolumn{2}{c}{CPU} & \multicolumn{2}{c}{RecErr} & \multicolumn{2}{c}{Residual} \\ \cmidrule(lr){4-5}  \cmidrule(lr){6-7}   \cmidrule(lr){8-9}
			  &  &   &    SPGL1 &  MBA$_{\ell_1/\ell_2}$  &    SPGL1 &  MBA$_{\ell_1/\ell_2}$  &  SPGL1 &  MBA$_{\ell_1/\ell_2}$ \\ \hline
              8 &   5 &   2 &   0.07 &   0.13 (  0.20)  & 3.2e-02  & 2.3e-03 & -4e-05 & -1e-13    \\
              8 &   5 &   3 &   0.06 &   0.14 (  0.20)  & 3.2e-03  & 6.8e-04 & -4e-05 & -2e-11    \\
              8 &  15 &   2 &   0.08 &   3.92 (  4.01)  & 4.7e-01  & 1.5e-01 & -9e-05 & -7e-13    \\
              8 &  15 &   3 &   0.11 &  31.46 ( 31.58)  & 3.8e-01  & 5.3e-02 &  2e-02 & -5e-11    \\ \hline
             12 &   5 &   2 &   0.06 &   2.26 (  2.32)  & 1.4e-01  & 3.6e-02 & -3e-04 & -8e-13    \\
             12 &   5 &   3 &   0.08 &   4.05 (  4.14)  & 6.0e-02  & 3.8e-03 &  1e-04 & -7e-11    \\
             12 &  15 &   2 &   0.09 &   8.32 (  8.41)  & 5.2e-01  & 2.0e-01 & -1e-04 & -1e-12    \\
             12 &  15 &   3 &   0.11 & 403.80 (403.91)  & 5.2e-01  & 1.5e+00 &  6e-02 & -3e-10    \\ \hline
			\end{tabular}
		}
	\end{center}
\end{table}



\end{document}